\documentclass[11pt]{amsart}
\scrollmode
\usepackage{amsmath, amsthm, amssymb,amsfonts,mathtools}
\usepackage{verbatim,geometry,graphicx,multirow,array}
\usepackage{hyperref}
\usepackage{enumerate}
\usepackage{xcolor}
\usepackage{xy,psfrag}
\usepackage{datetime}
\usepackage{refcount}

\newcommand{\Real}{\operatorname{Re}}
\newcommand{\Img}{\operatorname{Im}}
\newcommand{\inv}{^{-1}}
\newcommand{\dfn}{\vcentcolon =}
\newcommand{\norm}[1]{\left\Vert #1 \right\Vert}
\newcommand{\lam}{\lambda}
\newcommand{\eps}{\varepsilon}
\newcommand{\ddt}{\frac{d}{dt}}
\newcommand{\pardt}{\frac{\partial}{\partial t}}
\newcommand{\bigo}{\mathcal{O}}
\newcommand{\Lam}{\Lambda}

\newcommand{\trace}{\operatorname{tr}}

\newcommand{\R}{{\mathbb R}}
\newcommand{\C}{{\mathbb C}}

\newcommand{\Cnxn}{{\C^{n\times n}}}
\newcommand{\Ctxt}{{\C^{2\times 2}}}
\newcommand{\bmat}[1]{ \begin{bmatrix}#1\end{bmatrix}}

\newcommand{\cont}{{\mathcal C}}
\newcommand{\diag}{\operatorname{diag}}

\newcommand{\Cn}{{\C^n}}
\newcommand{\abs}[1]{\left| #1 \right|}

\renewcommand{\ss}{\scriptstyle}
\def\sddots{\mathinner{\raise3pt\vbox{\hbox{$\ss .$}}
		\raise1.5pt\hbox{$\ss .$}\hbox{$\ss .$}}}

\theoremstyle{plain}
\newtheorem{thm}{Theorem}[section]
\newtheorem{lem}[thm]{Lemma}
\newtheorem{cor}[thm]{Corollary}

\theoremstyle{definition}
\newtheorem{rem}[thm]{Remark}

\newtheorem{defn}[thm]{Definition}
\newtheorem{exm}[thm]{Example}
\newtheorem{agree}[thm]{Agreement}

\begin{document}
%\linenumbers

\title{Generic Cuspidal Points and Their Localization}
%\thanks{Accepted for publication in the SIAM Journal on Matrix Analysis and Applications (SIMAX). This is the author-accepted manuscript (pre-production version).}}

\author[Dieci]{Luca Dieci}
\address{School of Mathematics, Georgia Institute of Technology,
Atlanta, GA 30332 U.S.A.}
\email{luca.dieci@math.gatech.edu}
\author[Pugliese]{Alessandro Pugliese}
\address{Deparment of Mathematics, University of Bari Aldo Moro, Via Orabona 4, 70125 Bari, Italy}
\email{alessandro.pugliese@uniba.it}
\thanks{
The authors gratefully acknowledge the School of Mathematics of Georgia Tech for hosting the visit of the second author during
the Academic Year 2023-24. The work has been partially supported by GNCS-INdAM and by 
the PRIN2022PNR P2022M7JZW 
``SAFER MESH - Sustainable mAnagement oF watEr Resources ModElS and numerical MetHods'' research grant.}

\subjclass{15A18, 15A20, 15A23, 15A99}

%\null\hfill {\fontsize{10}{10pt} \selectfont Version of \today }

\keywords{Matrices depending on parameters, cuspidal points,
coalescence of eigenvalues, exceptional points}

\begin{abstract}
In this work we consider generic coalescing of eigenvalues of smooth complex valued matrix functions depending on 2 parameters.  We call generic cuspidal points the  parameter values where eigenvalues coalesce and we discuss the relation between cuspidal points and the closely related exceptional points studied in the literature.   By considering loops  in parameter space enclosing the cuspidal points, we rigorously prove when there is a phase accumulation for the eigenvectors and further detail how, by looking at the periodicity of the eigenvalues along the loop, and/or by looking at the aforementioned phase accumulation, one may be able to localize generic cuspidal points.
\end{abstract}

\maketitle

\pagestyle{myheadings}
\thispagestyle{plain}
\markboth{L.~Dieci, A.~Pugliese.}{Cuspidal points}

\noindent\textit{Accepted for publication in the SIAM Journal on Matrix Analysis and Applications.}

\bigskip

\noindent{\bf Notation.} 
Throughout the paper, we identify $\C$ with $\R^2$ in the usual way, i.e. $z=\Real(z)+i\Img(z)\equiv(\Real(z),\Img(z))\in\R^2$, and do similarly for a mapping $f:\C\rightarrow\C$: $f(z)\equiv f(\Real(z),\Img(z))=(\Real(f(z)),\Img(f(z)))$. For a vector $v\in\Cn$, $v^*$ indicates its conjugate transpose. The eigenvalues of a matrix $A\in\Cnxn$ are denoted by $\lam_1(A), \ldots, \lam_n(A)\in\C$, or just $\lam_1, \ldots, \lam_n$ if this doesn't cause any confusion. The set $\{\lam_1(A), \ldots, \lam_n(A)\}$ is called the \emph{spectrum} of $A$. The spectrum of $A$ is said to be \emph{simple} if it consists of $n$ distinct numbers; otherwise, it is said to be \emph{degenerate}. Similarly, a matrix $A$ is degenerate if it has degenerate spectrum, and an eigenvalue $\lam$ is degenerate if it occurs with algebraic multiplicity larger than 1. 
We write $A\in\cont^k(\Omega,\Cnxn)$ to indicate that $A$ has $k$ continuous derivatives at any point in $\Omega$. The symbol $\Omega$ will generally indicate a non-empty open convex subset of $\R^2$ or of $\C$. Points in $\Omega$ will be indicated by $\xi=(x,y)$ or $\xi=x+iy$. The vector norm will always be the Euclidean norm.

\section{Introduction, scope, and background}
In this work, we consider a matrix valued function $A\in\cont^k(\Omega,\Cnxn)$, $k\ge 1$, and are interested in giving rigorous results about detection and localization of parameter values $\xi\in \Omega$ where the eigenvalues of $A$ coalesce; these parameter values will be called {\emph{coalescing points}}.  Our interest is when $A$ has no special structure, in particular it is not a Hermitian matrix, a situation which has  received a lot of attention and where the celebrated ``non-crossing rule'' of Von Neumann and Wigner, see \cite{VonNeuWig}, clarifies that (within the class of generic smooth Hermitian matrices) one needs two (real) parameters for eigenvalues coalescence to occur.  Also in the present, non-Hermitian case, generically one needs two parameters for a coalescence to occur; see Corollary \ref{cor:keller}.   But, the nature of the singularity of the eigenvalues at a coalescence point is quite different from the double cone structure associated to coalescence in the Hermitian case, and it is akin to that of a square root singularity, hence the name of {\emph{cuspidal point}} which we are adopting in the present case; see Figure \ref{fig:cusp_surf}.

\begin{figure}
    \centering
    \includegraphics[width=0.75\linewidth]{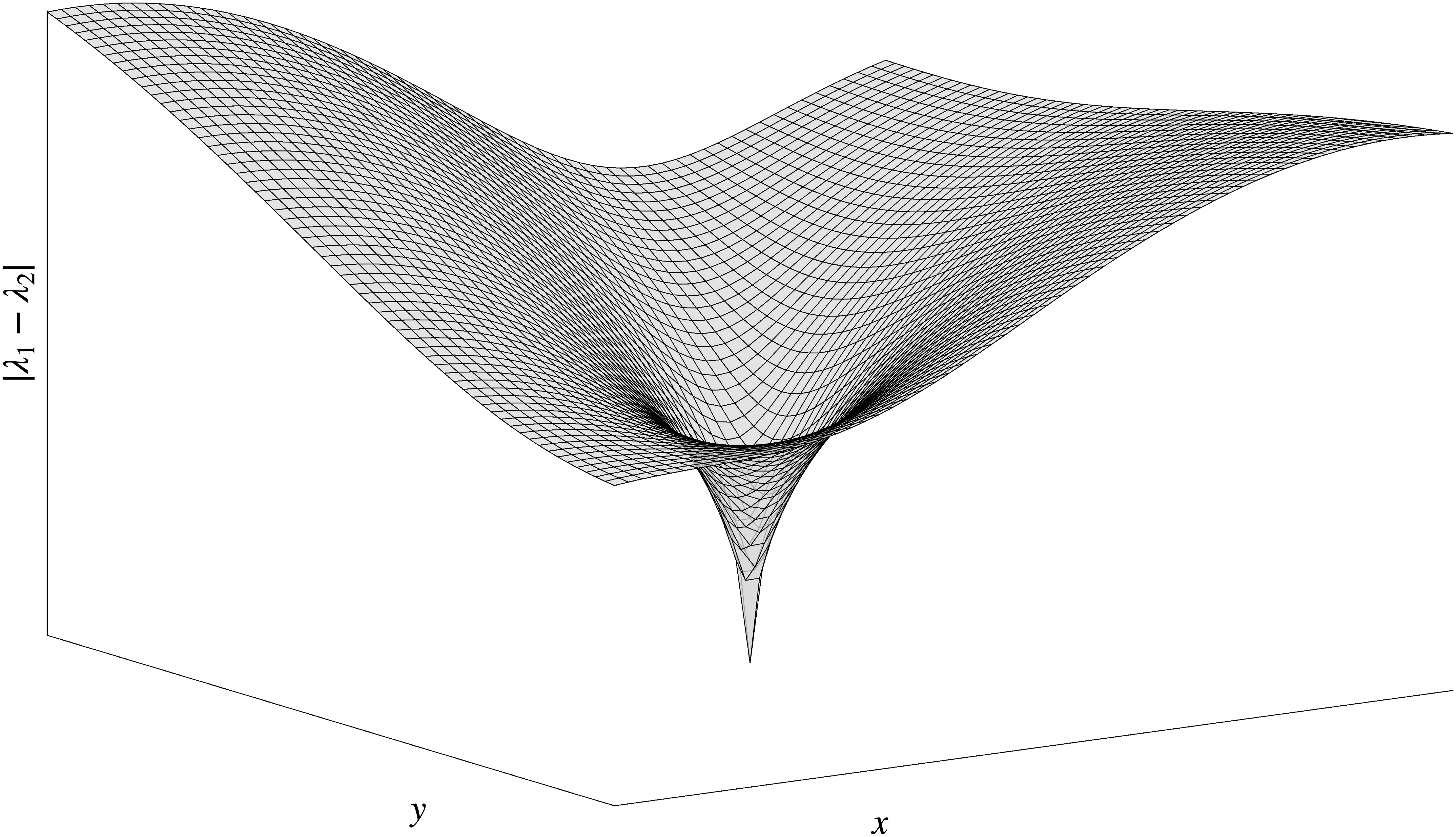}
    \caption{Generic behavior of the modulus of the difference between two coalescing eigenvalues near a cuspidal point.}
    \label{fig:cusp_surf}
\end{figure}

Our main results stem from rigorously proving that if one covers a loop containing a generic coalescing point, then there is a nontrivial phase accumulation for the eigenvectors along the loop.  We further give new results on the precise rate at which the phase accumulates when we take a loop containing the coalescing point, and let the loop shrink towards the coalescing point, and further contrast this asymptotic behavior to that one has when the loop does not contain a coalescing point.  Finally, we provide generalization of our results to the case of the loop enclosing several generic coalescing points.
Further, by monitoring the periodicity of the eigenvalues as we cover a loop, and/or by looking at the mentioned phase accumulation, we explain how one can detect and localize cuspidal points.

In the Physics literature, see references below, what we call generic cuspidal point is typically called {\em{exceptional point}}, and there are several works about exceptional points of non-Hermitian 2d-Hamiltonians (which, in our language, give $2\times 2$ complex valued matrices $A$), where the above mentioned phase accumulation has been observed.  A significant, but likely incomplete, literature review follows.  For non-Hermitian Hamiltonians, the works \cite{Bhandari, Garrison} were the first instances where it was observed that the geometrical phases associated with cyclic unitary time evolutions in quantum mechanics are replaced by complex geometrical multipliers.  The work \cite{Dembowski} reports on an experimental setting where exceptional points are encircled in the laboratory, and the authors recognized the square-root-like singularities of the eigenvalues as function of a complex interaction parameter as a signature of the phenomenon.  The later effort \cite{BerryUzdin} considers a simplified model for cycling around loops that enclose or exclude an isolated degeneracy of a non-Hermitian Hamiltonian.  The work \cite{Doppler} also observes the phase accumulation we mentioned above, and the recent work \cite{Deng} observes how the presence of an exceptional point impacts the dynamics of magnetic bilayers in the vicinity of an exceptional point.  Finally, in a series of works, \cite{Ryu1, Ryu2, Ryu3}, Ryu and coworkers tackle the case of multiple exceptional points. Yet, in spite of these works, we did not find a complete, general, and rigorous,  mathematical explanation of the phenomenon, and this is the main reason for our work.

\begin{rem}\label{HermCase}
	As already mentioned, the case of coalescing eigenvalues of parameter dependent matrices has been studied before in the self-adjoint case, and its mathematical development has reached a fairly mature stage, as have computational techniques to locate parameter values where eigenvalues coalesce.  E.g., see our own works \cite{DiPuSVD, DP12herm, DiPaPu_simax2013} as well as the more recent works \cite{BerkoPar, BerkoZel}, and references there.   But, the Hermitian case is theoretically much simpler in comparison to the non-Hermitian case considered in the present work, in the end because in the present case there is no inherent ordering of the eigenvalues (unlike when they are all real).  In the current work, we develop a general, and rigorous,  theoretical analysis of coalescing eigenvalues for non-Hermitian functions, and in forthcoming work we plan to consider computational methods to locate parameter values where such coalescings occur.
\end{rem}

A plan of our paper is as follows.  Below, in Section \ref{1dim} we will review results about smooth eigen-decompositions in the 1-parameter case, since this will play a key role later on, and in Section \ref{GenCoal} we will clarify why $A$ needs to depend on two real parameters in order to observe a coalescence of eigenvalues.  Then, in Section \ref{GCP} we present our main results, showing when and how it is possible to locate parameter values where coalescing eigenvalues occur.

\vskip 0.1in
\subsection{1-d case}\label{1dim}
Let $A\in\cont^k(\R,\Cnxn)$, $k\ge 1$, and assume that $A(t)$ has distinct eigenvalues for all  $t$. Then, one can write differential equations for the factors of its eigendecomposition 
(see \cite{DE99}).  That is,  for all $t\in \R$, one has $A(t)=V(t)\Lambda(t)V^{-1}(t)$, where 
$\Lam(t)=\diag(\lam_1(t),\ldots,\lam_n(t))$ is the matrix of eigenvalues and $V(t)$ is the matrix of eigenvectors, and further $\Lam$ and $V$ are as smooth as $A$ and satisfy the differential equations

\begin{equation}\label{eq:diffeq}
\begin{cases}
\dot\lambda_j = \big(V^{-1} \dot{A} V\big)_{jj}, & \text{for all } j = 1, \ldots, n, \\[1ex]
\dot{V} = V P, \quad \text{where } P_{jk} = \dfrac{\big(V^{-1} \dot{A} V\big)_{jk}}{\lambda_k - \lambda_j}, & \text{for all } j \ne k,\ j,k = 1, \ldots, n.
\end{cases}
\end{equation}

Note that from \eqref{eq:diffeq} the diagonal values of $P$ are undefined. This is just a manifestation of the degree of non-uniqueness possessed by the eigendecomposition of any matrix $A\in\Cnxn$ with simple spectrum: each eigenvector is unique up to multiplication by an arbitrary non-zero complex number. In fact, if $A$ is a function of $t$, and $A(t)=V(t)\Lam(t)V\inv(t)$ is an eigendecomposition having smooth factors, then the eigenvectors' matrix $V(t)$ is unique up to post-multiplication by a diagonal matrix function
\begin{equation}
\Theta(t)=\bmat{\rho_1(t)e^{i\theta_1(t)} & &\\
 & \ddots & \\
 & & \rho_n(t)e^{i\theta_n(t)}},
\end{equation}
where for all $j=1,\ldots,n$, and all $t$, $\rho_j>0$ and $\theta_j\in\R$,  are arbitrary smooth functions of $t$. The choice of the diagonal entries of $P$ completely determines how this non-uniqueness is resolved. A natural choice for $\Real\left(P_{jj}(t)\right)$ is the one for which the eigenvectors $v_j$ have constant length: $\norm{v_j(t)}_2=\norm{v_j(0)}_2$ for all $t$ and all $j$. This choice was considered in \cite{DE99}, and amounts to the following requirement:
\begin{equation}\label{eq:realP}
\Real(v_j^*v_jP_{jj}(t))=-\sum_{k\ne j}\Real(v_j^*v_kP_{kj}(t)), \text{ for all } t\in\R. 
\end{equation}
As for $\Img\left(P_{jj}(t)\right)$, we will make the following choice:
\begin{equation}\label{eq:imgP}
\Img(P_{jj}(t))=0, \text{ for all } t\in\R. 
\end{equation}
\begin{rem}
The choice given by \eqref{eq:imgP}
is called {\emph{parallel transport}}, e.g., see \cite{Simon, Bhandari},
and it is the mandatory choice to make in order to obtain Theorem \ref{thm:2GCPs_2x2} below.
\end{rem}

\subsection{Coalescence of eigenvalues}\label{GenCoal}
Having simple spectrum is a generic property for complex matrices: the set of complex $n\times n$ matrices having simple eigenvalues  is an open and dense subset of $\Cnxn$. Finer (and more helpful) information about when a matrix $A\in\Cnxn$ has degenerate spectrum is obtained by looking at its {\bf codimension}.  By codimension of a set of matrices we mean the number of (independent) conditions that a matrix must meet in order to be in that set. 
Next, we state a theorem that is a special case of one found in \cite{Keller08} (to which we refer for details), slightly rewritten to fit our purposes. We recall that $\Cnxn$ is a vector space of real dimension\footnote{Throughout, dimension and codimension of sets are always considered with respect to the real field.} $2n^2$.

\begin{thm}[{\cite[Theorem 7]{Keller08}}]\label{thm:keller} The set of matrices $A\in\Cnxn$ with $n-1$ distinct eigenvalues, with the degenerate eigenvalue occurring in one block of size $2$ in the Jordan canonical form of $A$, has dimension $2n^2-2$. The set of matrices $A\in\Cnxn$ with $n-1$ distinct eigenvalues, with the degenerate eigenvalue occurring in two blocks of size $1$ in the Jordan canonical form of $A$, has dimension $2n^2-6$. The set of matrices $A\in\Cnxn$ with fewer than $n-1$ distinct eigenvalues has dimension lower than $2n^2-4$.  
\end{thm}

Theorem \ref{thm:keller} has the following corollary. 

\begin{cor}\label{cor:keller} Coalescence of eigenvalues for a matrix $A\in\Cnxn$ is a codimension 2 phenomenon. Moreover, it is a generic property, within the set of  degenerate matrices in $\Cnxn$, to have $n-1$ distinct eigenvalues and to be non-diagonalizable.
\end{cor}

\begin{exm}
The conclusions in Theorem \ref{thm:keller} are easily appreciated 
by looking at a $2\times 2$ model. Consider
\begin{equation*}
A=\bmat{a & b \\ c & d}, \text{ with } a,b,c,d\in\C \ ,
\end{equation*}
so that 
\begin{equation*}
\lam_{1,2}(A)= \frac{a+d\pm \sqrt{(a-d)^2+4bc}}{2} \,\ \text{and}\,\ 
\lam_1(A)= \lam_2(A) \iff (a-d)^2+4bc=0. 
\end{equation*}
Note that the last equation amounts to two conditions, from which we deduce 
that coalescence of the eigenvalues of $A$ has codimension 2. 
If we require further $A$ to be diagonalizable, then we must have $a=d$ and $b=c=0$, which amounts to 6 conditions. Therefore, in the diagonalizable case, codimension of 
eigenvalues' coalescence is 6.
\end{exm}

Corollary \ref{cor:keller} implies that, in general, one needs to consider complex matrix valued functions of two real parameters\footnote{In this work, parameters are always real valued.} in order to observe coalescence of eigenvalues. It also says that, for 2 parameters matrix functions, coalescence is expected to occur at isolated points in parameter space and to persist under small perturbations, and that the degenerate eigenvalue is expected to occur in one Jordan block of size 2.
This is why, in this work, we consider complex matrix functions smoothly depending on two parameters: $A=A(x,y)\in\cont^k(\R^2,\Cnxn)$, with $k\ge 1$.

\begin{rem}
As already remarked, in the Physics literature (e.g., see \cite{Doppler})
parameter values where two eigenvalues coalesce forming a non-trivial Jordan
block are frequently called {\emph{exceptional points}}, but we caution that there is no consensus on this terminology.
In particular, the wording {\emph{exceptional point}} is used by Kato (see \cite[p.64]{Kato})
to indicate any parameter value where eigenvalues coalesce.  Thus, in Kato's language, the  problem $\bmat{0 & x+iy \\ x+iy & 0}$ has one exceptional point at the origin,
although at the origin the matrix is obviously diagonal.
In this work, we will actually adopt the naming of
{\emph{cuspidal points}} for the parameter values where a generic coalescing of eigenvalues occurs, since this name
reflects much better the type of singularity we have; see
Section \ref{GCP}.
\end{rem}

We conclude this section by giving a result, adapted from \cite{HS66}, which will be very useful in order to reduce consideration to the neighborhood of parameter values where eigenvalues coalesce.

\begin{thm}[Block decomposition]\label{thm:blockdec} Let $A\in\cont^k(\R,\Cnxn)$, $k\ge 0$, with $R\subset\R^2$ a rectangular set defined by $a_1\le x\le b_1$, $a_2\le y\le b_2$, with $a_i < b_i$, $i=1, 2$. Assume that the eigenvalues of $A$ can be labeled so that they belong to two disjoint sets $\Lam_1(\xi)=\{\lam_1(\xi), \ldots, \lam_p(\xi)\}$, $\Lam_2(\xi)=\{\lam_{p+1}(\xi), \ldots, \lam_n(\xi)\}$, $\Lam_1(\xi)\cap\Lam_2(\xi)=\emptyset$ for all $\xi=(x,y)\in R$. Then, there exists
an invertible $T\in\cont^k(R,\Cnxn)$ such that, for all $\xi \in R$, we have
\begin{equation}\label{eq:blockdec}
A(\xi)T(\xi)=T(\xi)\bmat{E_1(\xi) & 0\\ 0 & E_2(\xi)},
\end{equation}
where $E_1(\xi)$ is $p\times p$,  $E_2(\xi)$ is $(n-p)\times (n-p)$, and the eigenvalues of $E_i(\xi)$ are the values of $\Lam_i(\xi)$, for $i=1,2$. 
\end{thm}
\begin{rem}\label{RemsGingold}
Theorem \ref{thm:blockdec} can be further refined, up to a full diagonalization of $A$ in case all its eigenvalues are distinct in $R$. 
A similar theorem holds for functions of a single real variable (with $R$ replaced by an interval $[a, b]$), as well as for matrix functions defined on a plurirectangle in any dimension; see again \cite{HS66} for details.
\end{rem}
\color{black}

\section{Generic Cuspidal Points: Phase Accumulation of Eigenvectors
and Localization}\label{GCP}

In this section we give our main results, first by considering $2\times 2$ functions $A$, and then using the $2\times 2$ case as building block for generalizing to $n\times n$ functions.  Specifically: (i) we look at the type of eigenvalues coalescing we are interested in, (ii) we study the behavior of the phase of the eigenvectors associated to an isolated coalescing eigenvalue as we cover a loop enclosing the coalescing point, (iii) we study the asymptotic behavior of the phase of eigenvectors as the loop shrinks to a point, both when there is a coalescing point in the loop, and when there is not, and finally (iv) we consider  the case of multiple coalescing points inside the loop.

First, consider this simple example.

\begin{exm}\label{SqRt}
Let $A(\xi)=\bmat{0 & 1 \\ x+iy & 0}$.  Writing $z=x+iy$, the eigenvalues are $\lam_{1,2}=\pm \sqrt{z}$.  
Obviously, a double eigenvalue occurs only at $\begin{cases} x=0\\  y=0\end{cases}$, and we note that the Jacobian of this system is the identity matrix, clearly invertible.
\end{exm}
\color{black}

\begin{defn}[Generic Cuspidal Point]\label{def:cuspPoint} Let $A\in\cont^k(\Omega,\Cnxn)$, $k\ge 1$, with $\Omega\subset\R^2$ an open set. Let $\xi_0\in\Omega$ a point where $A$ is degenerate. We say that $\xi_0$ is a \emph{Generic Cuspidal Point} (GCP) for $A$ according to (i)-(ii) below.
\begin{enumerate}[(i)]
 \item If $n=2$, let $A(\xi)=\bmat{a(\xi) & b(\xi) \\ c(\xi) & d(\xi)}$ and set 
 \begin{equation}\label{eq:GCP_def_2x2}
 \Delta(\xi)\dfn (a(\xi)-d(\xi))^2+4b(\xi)c(\xi)\quad\text{and}\quad F(\xi)\dfn \bmat{\Real(\Delta(\xi)) \\ \Img(\Delta(\xi))},
\end{equation}
so that $F\in \cont^k$. Then, $\xi_0$ is a GCP for $A$ if
\begin{equation}\label{eq:gpc_2x2}
F(\xi_0)=\bmat{0 \\ 0}\quad \text{and}\quad DF(\xi_0)\text{ is invertible}.
\end{equation}
 \item If $n>2$, suppose there exists a rectangle $R\subset\Omega$ where the eigenvalues $\lam_i(\xi)$, $i=1,\dots,n$, of $A(\xi)$ can be grouped in two disjoint sets, $\{\lam_1(\xi),\lam_2(\xi)\} \cap \{\lam_3(\xi),\ldots,\lam_n(\xi)\}=\emptyset$ for all $\xi\in R$, and let $T(\xi)$ as in Theorem \ref{thm:blockdec} be such that
\begin{equation}
 T^{-1}(\xi)A(\xi)T(\xi)=\bmat{E_1(\xi) & 0\\ 0 & E_2(\xi)},
\end{equation}
where the eigenvalues of $E_1(\xi)$ are 
$\{\lam_1(\xi),\lam_2(\xi)\}$ and those of $E_2(\xi)$ are 
$\{\lam_3(\xi),\allowbreak\ldots,\allowbreak\lam_n(\xi)\}$. Then, $\xi_0$ is a GCP for $A$ if it is a GCP for 
$E_1$ according to part (i) of this definition, and the eigenvalues of $E_2(\xi)$ are simple for all $\xi\in R$.
\end{enumerate}	
\end{defn}
\begin{rem}
We stress that, geometrically, condition \eqref{eq:gpc_2x2} amounts to require transversal intersection of two planar curves.
\end{rem}
\begin{rem} According to Definition \ref{def:cuspPoint}, in a neighborhood of a GCP $\xi_0$, we can always assume that the coalescing eigenvalues are those labeled as $\lam_1$ and $\lam_2$.
\end{rem}

A fundamental observation is that (in general) it is not possible to label the eigenvalues as continuous functions of $\xi$ in a neighborhood of a GCP.  This is evident by considering the above Example \ref{SqRt} (and see \cite[p.~108]{Kato}).  Naturally, it is always possible to label the eigenvalues to be continuous functions of $\xi$ in case $A$ is Hermitian, simply because the eigenvalues are real in that case.  A very important result of Kato, see \cite[Theorem 5.2]{Kato}, states that it is always possible to label the eigenvalues so that they are continuous functions (even if they coalesce) when $A$ depends on just one real variable $t$; in particular, when $A$ is a periodic function of $t$ of period $1$, although in this $1$-periodic case the requirements of periodicity and continuity of the eigenvalues have to be carefully weighed against each other.  Our next result, Lemma \ref{lem:per1or2}, clarifies that, in this  $1$-periodic case, one may not be able to require both $1$-periodicity and continuity of the eigenvalues. This situation is in sharp contrast with the special case of $A$ being Hermitian, where it is always possible to enforce both continuity and periodicity -- though by forcing periodicity one may lose some degree of differentiability (see \cite{DE99}).  We emphasize that the next set of results about eigenvalues' periodicity for a periodic function $A$ are needed to monitor phase variations when one takes a smooth eigendecomposition along a loop in parameter space around a GCP $\xi_0$.

\begin{agree}\label{agreement} A recurring theme in the following results is that of restricting matrix-valued functions to plane loops parametrized by periodic functions. Before proceeding, henceforth we stipulate the following conventions.

\medskip
\begin{enumerate}[(i)]
\item When we say that a function $f$ is $p$-periodic, the value of $p$ is understood to be the minimal positive period.

\medskip
\item Following \cite{Younes2010}, by \emph{(parametrized) Jordan curve} we mean a continuous map $\gamma:\R\to\R^2$ such that:
\begin{enumerate}[a)]
	\item $\gamma$ is $1$-periodic;
	\item $\gamma$ is one-to-one on $[0,1)$.
\end{enumerate}
If $\gamma$ is $\cont^1$, we call it a \emph{smooth} Jordan Curve. We generally denote the image $\gamma([0,1))$ by $\Gamma$, and, with abuse of terminology, also refer to $\Gamma$ as a (smooth) Jordan curve.
Essentially, a Jordan curve $\Gamma$ is viewed as the image of a $1$-periodic parametrization $\gamma$, such that $\gamma(t)$ traces $\Gamma$ exactly one time as $t$ ranges from 0 to 1.
\end{enumerate}
\end{agree}

\begin{lem}\label{lem:per1or2} Let $A\in\cont^k(\R,\Ctxt)$, $k\ge 0$, be a $1$-periodic function and suppose $A(t)$ has simple spectrum for all $t\in\R$. If there exists a labeling of the eigenvalues of $A(t)$ that gives $2$-periodic continuous functions $\lam_1$ and $\lam_2$, then there cannot be a labeling of the eigenvalues, $\mu_1, \mu_2$, that is continuous and $1$-periodic.
\end{lem}
\begin{proof}
Observe that,
since the set of the eigenvalues must return into itself after one period, the two possibilities expressed in the statement of the Lemma 
are the only possibilities.	Moreover, they can be rewritten as follows. 
For the $2$-periodic continuous functions $\lam_1$ and $\lam_2$, they must satisfy
\begin{equation}\label{eq:2perLabel}
\lam_1(t+1)=\lam_2(t)\quad\text{and}\quad\lam_2(t+1)=\lam_1(t),\, {\text{for all}}\,\ t\in\R,
\end{equation}
whereas for the labeling $\mu_1, \mu_2$, we would have
\begin{equation}\label{eq:1perLabel}
	\mu_1(t+1)=\mu_1(t)\quad\text{and}\quad\mu_2(t+1)=\mu_2(t), \, {\text{for all}}\,\ t\in\R.
\end{equation}

Now, suppose that the two labelings in \eqref{eq:2perLabel} and \eqref{eq:1perLabel} are both valid.  Let
\begin{equation*}
 S_1=\{t\in\R: \lam_1(t)=\mu_1(t)\},\ S_2=\{t\in\R: \lam_1(t)=\mu_2(t)\}.
\end{equation*}
Note that $S_1$ and $S_2$ are closed, disjoint, and such that $S_1\cup S_2=\R$. Therefore, either $S_1=\R$ and $S_2=\emptyset$, or viceversa. In either case, $\lam_1$ would have period $1$, which contradicts \eqref{eq:2perLabel}.
\end{proof}

\begin{rem}
Henceforth, we will call {\bf{continuous labeling of the eigenvalues}} a labeling of the eigenvalues for which they are continuous functions.
Lemma \ref{lem:per1or2} clarifies that, for a continuous $1$-periodic $A$, having distinct eigenvalues for all $t$, the existence of a $2$-periodic continuous labeling of the eigenvalues rules out the existence of a continuous $1$-periodic one. In other words, if one can produce a $2$-periodic continuous labeling of the eigenvalues, then all continuous labelings must be $2$-periodic. This observation will play a key role in the proof of the next theorem.
\end{rem}\color{black}

\begin{thm}\label{thm:2cycle_2x2}
Let $A\in \cont^k(\Omega,\Ctxt)$, $k\ge1$, with $\Omega$ an open convex subset of $\R^2$. 
Let $\xi_0\in\Omega$ be a GCP, and suppose that $A(\xi)$ has distinct eigenvalues everywhere else in $\Omega$. Let $\Gamma$ be a Jordan curve encircling the point $\xi_0$, parametrized by $\gamma$ in such way\footnote{This is always possible for a continuous non-constant matrix-valued function, see Appendix~\ref{appendix:period}.} that both $\gamma(\cdot)$ and $A(\gamma(\cdot))$ are $1$-periodic. Finally, let $\lam_1(t),\lam_2(t)$ be a continuous labeling of the eigenvalues of $A(\gamma(t))$ along $\Gamma$. 
Then, for all $t\in\R$, we have
 \begin{equation}\label{eq:2cycle_2x2}
 \lam_1(t+1)= \lam_2(t)\quad\text{ and }\quad\lam_2(t+1)= \lam_1(t).
\end{equation}
\end{thm}
\begin{proof}  
Without loss of generality, we take $\xi_0$ to be the origin.
Let $F$ be given by \eqref{eq:GCP_def_2x2}.
By the Inverse Function Theorem, there exist an open neighborhood $U$ of $\xi_0$ and an open neighborhood 
$W\subset\R^2$ of the origin such that $F$ is one-to-one on $U$ and $F(U)=W$. 
Let $0<\sigma\le1$ be small enough so that $\sigma\Gamma\subset U$. Then, since $F$ is continuous and one-to-one on $U$,
$F(\sigma\Gamma)$ is a Jordan curve encircling the origin parametrized by $F(\sigma\gamma(t))$, which is 1-periodic and --without loss of generality-- can be written as
\begin{equation*}
	F(\sigma\gamma(t))=r(t)\bmat{\cos(\theta(t)) \\ \sin(\theta(t))},\quad t\in\R,
\end{equation*}
where $r$ and $\theta$ are continuous and satisfy $r(t+1)=r(t)>0$,  and $\theta(t+1)=\theta(t)+2\pi$.
A straightforward computation shows that the following expressions yield continuous 2-periodic parametrizations of the two eigenvalues of $A(\sigma\gamma(t))$ that satisfy \eqref{eq:2cycle_2x2}:
\begin{equation*}
	\begin{split}
		\tilde\lam_1(t)& ={\frac 12}\trace(A(\sigma\gamma(t)))+\frac 12\sqrt{r(t)}(\cos(\theta(t)/2)+i\sin(\theta(t)/2)), \\
		\tilde\lam_2(t)& =\frac 12\trace(A(\sigma\gamma(t)))-\frac 12\sqrt{r(t)}\,(\cos(\theta(t)/2)+i\sin(\theta(t)/2))\ .
	\end{split}
\end{equation*}
Now, consider the matrix-valued function 
$$M: (t,s)\in R=[0,1]\times[\sigma,1]\mapsto A(s\gamma(t))\in\Ctxt\ ,$$
which is continuous and has distinct eigenvalues for all $(t,s)\in R$. By Theorem \ref{thm:blockdec}, $M(t,s)$ has eigenvalues $E_1(t,s), E_2(t,s)$ which are continuous everywhere on $R$. Assume, without loss of generality, that $E_j(\cdot,\sigma)=\tilde\lam_j(\cdot)$, $j=1, 2$. Recall that these eigenvalues satisfy \eqref{eq:2cycle_2x2}. This property clearly carries by continuity up to the eigenvalues $E_1(\cdot,1), E_2(\cdot,1)$ of $M(\cdot,1)=A(\gamma(t))$. Finally, Lemma \ref{lem:per1or2} ensures that any continuous labeling of the eigenvalues of $A(\gamma(t))$ must satisfy \eqref{eq:2cycle_2x2}, and this concludes the proof.
\end{proof}

\begin{rem}
The essence of Theorem \ref{thm:2cycle_2x2} is that, upon continuation around a simple loop that encircles a GCP, the coalescing eigenvalues
undergo a \emph{$2$-cycle}\footnote{That is, the elements of the set ${\lambda_1(t), \lambda_2(t)}$ undergo a cyclic permutation as $t$ increases by one period. See \cite{Kato} for further details and for the general definition of $p$-cycles in this context.}. This is essentially a manifestation of a branch point for the multi-valued function $\xi\mapsto \sqrt{\Delta(\xi)}$, see \cite[p.65]{Kato}, which is the very reason precluding being able to label the eigenvalues as continuous functions of $\xi$ in a neighborhood of a GCP.
\end{rem}

\begin{rem}\label{rem:homotopy}
The assumption in Theorem \ref{thm:2cycle_2x2} of $\Omega$ being convex allows us to adopt the linear homotopy $s\Gamma$ with the guarantee to remain in $\Omega$.  At the same time, one may weaken the assumption on $\Omega$ asking it to be just open and simply connected and then adopt a homotopy argument like the one we adopted in \cite{DiPuSVD}; see \cite[Theorem 2.2 and Remark 2.5]{DiPuSVD}.  This extra layer of complexity notwithstanding, the key fact implied by Theorem \ref{thm:2cycle_2x2} is that the periodicity of the eigenvalues along a Jordan curve remains unchanged under any continuous deformation of the curve, provided that no eigenvalues coalesce during the deformation.
\end{rem}

Next, we look at results about phase accumulation of the eigenvectors as we complete a loop in parameter space enclosing -or not enclosing- a GCP.

\begin{lem}\label{lem:alpha_trace_det}
Let $A\in\cont^k(\R,\Cnxn)$, $k\ge 1$, be $1$-periodic and have distinct eigenvalues for all $t$. Let $A(t)=V(t)\Lam(t)V\inv(t)$ be a $\cont^k$ eigendecomposition of $A$ satisfying \eqref{eq:realP} and \eqref{eq:imgP}.  Let $\Phi(t)=\diag(e^{i\alpha_j(t)},j=1,\ldots,n)$, where $\alpha_j\in\cont^k(\R,\R)$ for all $j$, and let $\Pi$ be a permutation matrix such that
\begin{equation}\label{eq:alpha_trace_det}
    V(t+1)=V(t)\Pi\Phi(t), \text{ for all } t\in\R.
\end{equation}
Then, we have:
\begin{enumerate}[(i)]
    \item $\alpha_j(t)$ is constant for all $j$ and for all $t$;
    \item $\trace(P(V(t+1)))=\trace(P(V(t)))$ for all $t$, where $P(V(t))=V\inv(t)\dot V(t)$, see \eqref{eq:diffeq};
    \item $\det(V(t+1))=\det(V(t))$ for all $t$;
    \item for all $t$, $\sum\limits_{j=1}^n\alpha_j=0, \mod \pi$.
\end{enumerate}
\end{lem}
\begin{proof}
Differentiating \eqref{eq:alpha_trace_det}, we obtain
\begin{equation}\label{eq:alpha_trace_det_proof}
P(V(t+1))=\Phi\inv(t)\Pi^TP(V(t))\Pi\Phi(t)+i\diag(\dot\alpha_j(t),j=1,\ldots,n).
\end{equation}
Since $\Phi\inv(t)\Pi^TP(V(t))\Pi\Phi(t)$ and $P(V(t))$ share the same diagonal entries (up to permutation), using \eqref{eq:imgP}, we have that $\dot\alpha_j(t)=0$ for all $j=1,\ldots,n$, and therefore \emph{(i)} holds. Given \emph{(i)}, \emph{(ii)} follows trivially from \eqref{eq:alpha_trace_det_proof}.
Now, Liouville's formula gives:
\begin{equation*}
\ddt \log(\det(V(t)))=\trace(P(V(t)))\ .
\end{equation*}
Therefore, we can write
\begin{equation*}
\det(V(t))=\det(V(0))\,e^{\int\limits_0^t \trace(P(V(\tau)))\,d\tau}\ .
\end{equation*}
Being $\trace(P(V(t)))$ $1$-periodic, $\det(V(t))$ must be either $1$-periodic or unbounded. But $\det(V(t))$ cannot be unbounded because \eqref{eq:realP} implies that
\begin{equation*}
\abs{\det(V(t))}\le \norm{v_1(t)}_2\cdots\norm{v_n(t)}_2=\norm{v_1(0)}_2\cdots\norm{v_n(0)}_2\,,
\end{equation*}
for all $t\in\R$. Therefore, $\det(V(t))$ is $1$-periodic, and hence \emph{(iii)}. Finally, we obtain \emph{(iv)} by taking determinants in \eqref{eq:alpha_trace_det} and using \emph{(iii)}.
\end{proof}
Using the above Lemma, we now show that
\begin{thm}\label{thm:phasesum_2x2}
Let $A\in\cont^k(\R,\Ctxt)$, $k\ge 1$, be $1$-periodic and with distinct eigenvalues for all $t$. Let $A(t)=V(t)\Lam(t)V\inv(t)$ be a $\cont^k$ eigendecomposition of $A$ satisfying \eqref{eq:realP} and \eqref{eq:imgP}. Moreover, let $\Lam(t)=\diag(\lam_1(t),\lam_2(t))$ and suppose that $\lam_1(t)$ and $\lam_2(t)$ satisfy \eqref{eq:2perLabel}, that is:
\begin{equation*}
\lam_1(t+1)=\lam_2(t)\quad\text{and}\quad\lam_2(t+1)=\lam_1(t)\ ,\,\ \text{for all}\,\  t\in\R. 
\end{equation*}
Then, we have
\begin{equation}\label{eq:mainres_2x2}
V\inv(t)V(t+1)=\bmat{0 & e^{i\alpha_2} \\ e^{i\alpha_1} & 0} ,\,\ \text{for all}\,\  t\in\R ,
\end{equation}
where $\alpha_j\in\R,\ j=1,2$, satisfy
\begin{equation}\label{eq:sum_is_pi_2x2}
\alpha_1+\alpha_2=\pi, \mod 2\pi.
\end{equation}
\end{thm}
\begin{proof} 
Since $V(t+1)\Lambda(t+1)V^{-1}(t+1)=V(t)\Lambda(t)V^{-1}(t)$, then
$(V^{-1}(t)V(t+1))\bmat{\lam_1(t+1) & 0 \\ 0 & \lam_2(t+1)}=\bmat{\lam_1(t) & 0 \\ 0 & \lam_2(t)}(V^{-1}(t)V(t+1))$ and thus
\eqref{eq:2perLabel}  implies that $(V^{-1}(t)V(t+1))\bmat{\lam_2(t) & 0 \\ 0 & \lam_1(t)}=\bmat{\lam_1(t) & 0 \\ 0 & \lam_2(t)}(V^{-1}(t)V(t+1))$ and, since the eigenvalues are distinct, then $V^{-1}(t)V(t+1)$ has the form $\bmat{0 & \texttt{x} \\ \texttt{x} & 0}$.  But, because of \eqref{eq:realP}, the columns of $V(t)$ have constant length, and therefore
\begin{equation}\label{eq:V_of_t_plus_one}
V(t+1)=V(t)\bmat{0 & e^{i\alpha_2(t)} \\ e^{i\alpha_1(t)} & 0}
\end{equation}
for some $\cont^k$ real valued functions $\alpha_j(t),\ j=1,2$, and for all $t\in\R$. Then, Lemma \ref{lem:alpha_trace_det} implies that $\alpha_j(t),\ j=1,2$, are constant with respect to $t$. Taking determinants in \eqref{eq:V_of_t_plus_one}, and using \emph{(iii)} of Lemma \ref{lem:alpha_trace_det}, we have $-e^{i(\alpha_1+\alpha_2)}=1$ and this concludes the proof.
\end{proof}
The following result follows at once from Theorems \ref{thm:2cycle_2x2} and \ref{thm:phasesum_2x2} and does not require a proof.
\begin{thm}\label{thm:mainres_2x2}
Let $A\in\cont^k(\Omega,\Ctxt)$, $k\ge1$, where $\Omega$ is open and convex subset of $\R^2$. Let $\xi_0\in\Omega$ be a GCP, and suppose that $A(\xi)$ has distinct eigenvalues everywhere else in $\Omega$. Let $\Gamma$ be a smooth Jordan curve encircling the point $\xi_0$, parametrized by $\gamma$ in such way that both $\gamma(\cdot)$ and $A(\gamma(\cdot))$ are $1$-periodic. Finally, let $A(\gamma(t))=V(t)\Lam(t)V\inv(t)$ be a $\cont^k$ eigendecomposition of $A$ satisfying \eqref{eq:realP} and \eqref{eq:imgP}.
Then, for all $t\in\R$, we have:
\begin{enumerate}[(i)]
 \item $\lam_1(t+1) = \lam_2(t)$, $\lam_2(t+1)= \lam_1(t)$;
 \item $V\inv(t)V(t+1) = \bmat{0 & e^{i\alpha_2} \\ e^{i\alpha_1} & 0}$, with $\alpha_1+\alpha_2=\pi,\, \mod 2\pi$.\qed
\end{enumerate}
\end{thm}
\begin{rem} We note that, with the same hypotheses and notations of Theorem \ref{thm:mainres_2x2}, we have
\begin{align*}
    \lam_1(t+2) & =\lam_1(t),\quad \lam_2(t+2)=\lam_2(t),\\
    V(t+2) & =V(t)
    \bmat{e^{i(\alpha_1+\alpha_2)} & 0 \\ 0 & e^{i(\alpha_1+\alpha_2)}}=-V(t),
\end{align*}
and therefore
\begin{equation*}
    \Lam(t+4)=\Lam(t),\quad V(t+4)=V(t).
\end{equation*}
In other words, if a loop $\Gamma$ encloses exactly one GCP (which is the only point of coalescence of eigenvalues inside the region bounded by $\Gamma$), the smooth decomposition of $A(\gamma(t))$ around $\Gamma$ satisfying \eqref{eq:realP} and \eqref{eq:imgP} will have period 4. In contrast, when
$A$ is real symmetric --another notable case where coalescence of eigenvalues has codimension 2--  the factors of the smooth Schur (eigen)decomposition around a loop have period $1$ or $2$, again depending on the presence of eigenvalue coalescences inside the loop; see \cite{DiPuSVD}.
\end{rem}
The next result contrasts the previous Theorem \ref{thm:mainres_2x2} to the case when the eigenvalues of $A$ are distinct everywhere in $\Omega$, in which case they are $1$-periodic when $A$ is restricted to $1$-periodic loops.
\begin{thm}\label{thm:simple_eig} Let $A\in\cont^k(\Omega,\Cnxn)$, $k\ge1$, with $\Omega$ an open and convex subset of $\R^2$. Suppose that $A(\xi)$ has distinct eigenvalues everywhere in $\Omega$. Let $\Gamma$ be a smooth Jordan curve encircling the point $\xi_0$, parametrized by $\gamma$ in such way that both $\gamma(\cdot)$ and $A(\gamma(\cdot))$ are $1$-periodic. Finally, let $A(\gamma(t))=V(t)\Lam(t)V\inv(t)$ be a $\cont^k$ eigendecomposition of $A(\gamma(\cdot))$ satisfying \eqref{eq:realP} and \eqref{eq:imgP}.
Then, for all $j=1,\ldots,n$ and for all $t\in\R$, we have:
\begin{enumerate}[(i)]
 \item $\lam_j(t+1) = \lam_j(t)$;
 \item $v_j(t+1) = e^{i\alpha_j}\,v_j(t)$ for some (constant)  $\alpha_j\in\R$.
\item $\sum\limits_{j=1}^n \alpha_j=0, \, \mod \pi$.
\end{enumerate}
\end{thm}
\begin{proof}
%\AP{
First, let $R \subset \Omega$ be a rectangular region. By Theorem~\ref{thm:blockdec}, there exists $T\in\cont^k(\Omega,\Cnxn)$ invertible and such that
$$
T^{-1}(\xi)A(\xi)T(\xi) = \operatorname{diag}(E_j(\xi),\ j=1,\ldots,n)
$$
for all $\xi \in R$, where each $E_j(\xi)$ is an eigenvalue of $A(\xi)$ and is a $\mathcal{C}^k$ function of $\xi$.\newline
Now, consider a homotopy that continuously deforms $\Gamma$ into $\tilde{\Gamma}$, with $\tilde{\Gamma}$ parametrized by $\tilde{\gamma}$ and fully contained within $R$. The eigenvalues $E_j(\tilde{\gamma}(t))$ of $A(\tilde{\gamma}(t))$ are clearly 1-periodic, and, by Remark~\ref{rem:homotopy}, the same must hold for the eigenvalues $E_j(\gamma(t))=\lambda_j(t)$ of $A(\gamma(t))$. This establishes point \emph{(i)} for each $\lambda_j$.\newline
Next, let $v_j(t)$ denote the $j$-th column of $V(t)$. Given that the eigenvalues $\lambda_j(t)$ of $A(\gamma(t))$ are 1-periodic, it follows that, for each $j=1,\ldots,n$,
$$
v_j(t+1) = e^{i\alpha_j(t)}v_j(t),
$$
for some real-valued $\mathcal{C}^k$ function $\alpha_j(t)$. The fact that each $\alpha_j$ is constant follows as for point \emph{(i)} of Lemma~\ref{lem:alpha_trace_det}.\newline
Finally, statement \emph{(iii)} follows from equation~\eqref{eq:alpha_trace_det}, with $\Pi$ being the identity matrix.
\end{proof}
\begin{rem}
A notable consequence of Theorem \ref{thm:simple_eig} is the following.  Suppose that $A$ has distinct eigenvalues along a 1-periodic loop, and we compute these eigenvalues along the loop enforcing \eqref{eq:realP} and \eqref{eq:imgP}.  Then, if any of these eigenvalues fail to be 1-periodic, there must be a point inside the loop where some eigenvalues of $A$ coalesce. This fact gives a sufficient condition to detect coalescence of eigenvalues of non-Hermitian matrices.
\end{rem}
\begin{rem}\label{rem:phase_accrued_around_a_loop} Under the assumptions of Theorem~\ref{thm:simple_eig}, point~\emph{(ii)} asserts that each eigenvector $v_j(t)$ of $A(\gamma(t))$ acquires a phase factor $e^{i\alpha_j}$ as $t$ completes one period. We will refer to $\alpha_j$ as the \emph{phase accrued by $v_j$} along the closed curve $\Gamma$.
\end{rem}

The following result is useful to obtain explicit formulas for the phases accrued by eigenvectors around loops, and asymptotics as the loops shrink down to points, and it is an extension of Lemma 2.1 of \cite{DP12herm} to the case of non-Hermitian matrix functions.
\begin{thm}\label{thm:phase_formula}
Let $A\in\cont^k(\R,\Cnxn)$, $k\ge 1$, be $1$-periodic and have distinct eigenvalues for all $t\in\R$. For a fixed labeling of the eigenvalues, let $\Lambda(t)={\diag}(\lam_i(t),\allowbreak i=1,\dots, n)$, for all $t$.  Let $A(t)=V(t)\Lambda(t)V\inv(t)=\widetilde V(t)\Lambda(t)\widetilde V\inv(t)$ be two $\cont^k$ eigendecompositions of $A(\cdot)$ such that
\begin{enumerate}[(i)]
    \item $V$ satisfies \eqref{eq:imgP};
    \item $\widetilde V(1)=\widetilde V(0)\Pi$ for a permutation $\Pi$ (the same as in \eqref{eq:phase_formula_statement} below);
    \item $\widetilde V(0)=V(0)\diag(c_i, \, c_i>0,\, i=1,\dots, n)$.
\end{enumerate}
Let $W\dfn V\inv$ and $\widetilde W\dfn\widetilde V\inv$. Write $V$ and $\widetilde V$ partitioned by columns, e.g. $V=\bmat{v_1, \ldots, v_n}$, and $W$ and $\widetilde W$ partitioned by rows, e.g. $W=\bmat{w_1^T \\ \vdots \\ w_n^T}$. Finally, let $\Phi=\diag(e^{i\alpha_j},\alpha_j\in \R, j=1,\ldots,n)$ and $\Pi$ be a permutation matrix as in \eqref{eq:alpha_trace_det} of Theorem \ref{lem:alpha_trace_det}, that is such that
\begin{equation}\label{eq:phase_formula_statement}
V\inv(0)V(1)=\Pi\Phi\ .
\end{equation}
Then, for all $j=1,\ldots,n$,
\begin{equation}\label{eq:phase_formula}
\alpha_j=-\int\limits_0^1\Img\Big(\widetilde w_j(t)^T \dot{\widetilde{v}}_j(t)\Big)\, dt, \mod 2\pi.
\end{equation}
\end{thm}
\begin{proof}
For any $j=1,\ldots,n$, we have
\begin{equation}\label{eq:phase_formula_proof_1}
    v_j(t)=\rho_j(t)e^{i\varphi_j(t)}\widetilde v_j(t),
\end{equation}
for some smooth real scalar valued functions $\rho_j> 0$ and $\varphi_j$. Differentiating and using \eqref{eq:phase_formula_proof_1} and $\widetilde w_j^T(t)=\rho_j(t)e^{i\varphi_j(t)}w_j^T(t)$, we obtain
\begin{equation}\label{eq:phase_formula_proof_2}
    w_j^T(t)\dot v_j(t)= \dfrac{\dot\rho_j(t)}{\rho_j(t)}+i\dot\varphi_j(t)+\widetilde w_j^T(t) \dot{\widetilde{v}}_j(t).
\end{equation}
Taking imaginary parts in \eqref{eq:phase_formula_proof_2}, using \eqref{eq:imgP}, and integrating, we get
\begin{equation*}
    \varphi_j(1)-\varphi_j(0)=-\int\limits_0^1\Img\Big(\widetilde w_j^T(t) \dot{\widetilde{v}}_j(t)\Big)\, dt.
\end{equation*}
The sought result follows from the realization that, because of \emph{(iii)}, $\varphi_j(0)$ can be taken to be zero, and that, because of \emph{(ii)}, \eqref{eq:phase_formula_proof_1}, and \eqref{eq:phase_formula_statement}, we have $\varphi_j(1)=\alpha_j, \mod 2\pi$.
\end{proof}
\begin{rem}\label{rem:singleval_eigenvec} Note that, if one can write
\begin{equation*}
    A(\xi)\widetilde V(\xi)=\widetilde V(\xi)\Lam(\xi),\ \xi\in\Omega\subset\R^2, 
\end{equation*}
where $\Lam(\xi)=\diag(\lam_j(\xi),j=1,\ldots,n)$, and takes a $\cont^k$, $k\ge 1$, and $1$-periodic $\gamma:t\in\R\mapsto\gamma(t)\in \Omega$, such that $\widetilde V(\gamma(t))$ is $\cont^k$ and invertible for all $t\in\R$, then $\widetilde V(\gamma(t))$ will automatically satisfy \emph{(ii)} of Theorem \ref{thm:phase_formula}. This observation will be useful in obtaining the next two results.
\end{rem}
\begin{rem}\label{rem:complexConj_negativePhase}
    It follows directly from \eqref{eq:phase_formula} that --as long as the eigendecomposition evolves according to \eqref{eq:imgP}-- the phases accrued by the eigenvectors of $\overline{A}$ around a loop are the negatives of those accrued by the eigenvectors of $A$ (when the loop is traversed with the same orientation). This observation will be used in Example~\ref{exm:phase_zero}.
\end{rem}
Below, we give results on the asymptotic of the phases accrued by the eigenvectors as loops in parameter space shrink down to points, that completely clarify the difference between the case when the loop encloses a GCP and the case when it does not. First, we consider the case of the loop not containing a GCP.
\begin{thm}\label{thm:loops2point_simple_eig}
    With the same notation and hypotheses as in Theorem \ref{thm:simple_eig}, consider the family of smooth Jordan curves $\Gamma_s$ parametrized by $s\gamma(t)$, where $s\in[0,1]$. For each $j=1,\ldots,n$, let $\alpha_j(s)$ be the phase accrued by the eigenvector $v_j$ along $\Gamma_s$. Then,
    \begin{equation}\label{eq:loops2point_simple_eig}
        \alpha_j(s)=\bigo(s),\mod 2\pi,\ \text{ as } s\rightarrow 0.
    \end{equation}
\end{thm}
\begin{proof} Since $A$ has distinct eigenvalues everywhere inside the region bounded by $\Gamma$, we can select, see Remark~\ref{rem:singleval_eigenvec}, the matrix of eigenvectors that do not accrue a phase around the loop, and thus we can write
$$A(s\gamma(t)) = \widetilde V(s\gamma(t))\Lambda(s\gamma(t))\widetilde V\inv(s\gamma(t))$$
smoothly with respect to $t$ and $s$, where we have $\widetilde V(s\gamma(1))=\widetilde V(s\gamma(0))$ for all $s\in[0,1]$. Therefore, we can use \eqref{eq:phase_formula} and write
\begin{equation*}
\alpha_j(s)=-\int\limits_0^1\Img\left(\widetilde w_j^T(s\gamma(t)) \pardt{\widetilde{v}}_j(s\gamma(t))\right)\,dt, \mod 2\pi,
\end{equation*}
for all $j=1,\ldots,n$, and $s\in[0,1]$. Applying the chain rule to $\pardt{\widetilde{v}}_j(s\gamma(t))$, dividing by $s\in(0,1]$, and taking the limit as $s\rightarrow 0$, we have
\begin{equation}
    \lim\limits_{s\rightarrow 0}\frac{\alpha_j(s)}{s}=\text{const.}
\end{equation}
where $\text{const.}$ depends on $\dot \gamma$, hence on the parametrization of $\gamma$, but not on $s$.
\end{proof}
The next result is when the loop encloses a GCP.
\begin{thm}\label{thm:loops2point_GCP}
    With the same notation and hypotheses as in Theorem \ref{thm:mainres_2x2}, consider the family of closed curves $\Gamma_s$ parametrized by $\gamma_s:t\in\R\mapsto \gamma_s(t)\dfn\xi_0+s(\gamma(t)-\xi_0)\in\Omega\subset\R^2$, where $s\in(0,1]$. 
    For $j=1,2$, let $\alpha_j(s)$ be the phase defined in point~(ii) of Theorem~\ref{thm:phasesum_2x2} when $\Gamma$ is replaced by $\Gamma_s$. Then,
    \begin{equation}\label{eq:loops2point_GCP}
        \alpha_1(s)=\alpha_2(s)=\pm\frac{\pi}{2}+\bigo(\sqrt{s}),\mod 2\pi,\ \text{ as } s\rightarrow 0.
    \end{equation}
\end{thm}
\begin{proof}
    Without loss of generality, we can take $A\in\cont^k(\Omega,\Ctxt)$ to have $0$-trace:
    \begin{equation*}
        A=\bmat{a(\xi) & b(\xi) \\ c(\xi) & -a(\xi)},\ \xi\in\Omega.
    \end{equation*}
    With $\Delta(\xi)\dfn a(\xi)^2+b(\xi)c(\xi)$, the eigenvalues of $A$ are $\lam_1(\xi)=\sqrt{\Delta(\xi)}$,  $\lam_2(\xi)=-\sqrt{\Delta(\xi)}$ (where, for consistency, the same branch of the square root is used in both expressions), and the corresponding eigenvectors are
    \begin{equation*}
        \widetilde v_1(\xi)=
        \bmat{a(\xi)+\sqrt{\Delta(\xi)} \\ c(\xi)},\ 
        \widetilde v_2(\xi)=
        \bmat{a(\xi)-\sqrt{\Delta(\xi)} \\ c(\xi)}.
    \end{equation*}
Note that, for $\xi\ne\xi_0$, we have that $\lam_1(\xi)\ne\lam_2(\xi)$ and $v_1(\xi), v_2(\xi)$ are linearly independent. We are, therefore, in the position of applying Theorem \ref{thm:phase_formula} (see Remark \ref{rem:singleval_eigenvec}), and direct computation yields, for $s>0$,
\begin{equation}\label{eq:phasediff_2x2}
    \alpha_1(s)-\alpha_2(s)=-\int\limits_0^1\Img\left(
    \frac{c(\gamma_s(t))\pardt a(\gamma_s(t))-a(\gamma_s(t))\pardt c(\gamma_s(t))}{c(\gamma_s(t))\sqrt{\Delta(\gamma_s(t))}}
    \right)\,dt, \mod 2\pi.
\end{equation}
Applying the chain rule, dividing by $\sqrt{s}\in (0,1]$, and taking the limit as $s\rightarrow 0$, we obtain
\begin{equation*}
    \lim\limits_{s\rightarrow 0}\frac{\alpha_1(s)-\alpha_2(s)}{\sqrt{s}}=\text{const.}
\end{equation*}
where $\text{const.}$ depends on $\dot \gamma$, hence on the parametrization of $\gamma$, but not on $s$.
Pairing this with $\alpha_1(s)+\alpha_2(s)=\pi,\mod 2\pi$ (see Theorem \ref{thm:mainres_2x2}), yields the sought result.
\end{proof}
\begin{rem}\label{rem:special2x2}
    Note that, if $A(\xi)=\bmat{0 & b(\xi) \\ c(\xi) & 0}$, then in Equation \eqref{eq:phasediff_2x2} we have precisely $\alpha_1(s)-\alpha_2(s)=0,\mod 2\pi$, for all $s\in(0,1]$, that is, in Equation \eqref{eq:loops2point_GCP}:
    \begin{equation*}
    \alpha_1(s)=\alpha_2(s)=\pm\dfrac{\pi}{2},\mod 2\pi,\ \text{ for all}\ s\in(0,1].
    \end{equation*}
\end{rem}
Our next task is to consider the case of a loop in parameter space enclosing several GCPs and to see how the eigenvectors accrue a phase in this case. To deal with this situation, we will need a preliminary result showing that, under conditions \eqref{eq:realP} and \eqref{eq:imgP}, the smooth eigendecomposition of $A(t)$ solution of \eqref{eq:diffeq} enjoys a reversibility property: if we evolve the eigendecomposition according to \eqref{eq:diffeq}, imposing \eqref{eq:realP} and \eqref{eq:imgP}, say from $A(0)$ to $A(1)$ and then back to $A(0)$, then there is no phase accumulation for the eigenvectors from the beginning to the end of the path.

\begin{lem}\label{lem:back_and_forth}
    Let $A\in\cont^k([0,1],\Cnxn)$, $k\ge 1$, have distinct eigenvalues for all $t$. Let $\eta:t\in[0,1]\mapsto[0,1]$ be $\cont^k$, $k\ge 1$, and such that $\eta(1-t)=\eta(t)$, $\eta(0)=\eta(1)=0$, $\eta(1/2)=1$. Let $A(\eta(t))=V(t)\Lambda(t)V\inv(t)$ be an eigendecomposition satisfying \eqref{eq:realP} and \eqref{eq:imgP}. Then, $V(1)=V(0)$.
\end{lem}
\begin{proof}
    Since $A(\eta(\cdot))$ satisfies the hypothesis of Theorem \ref{thm:phase_formula}, with the notation of that Theorem we have $V\inv(0)V(1)=\Phi$, with $\Phi=\diag(e^{i\alpha_j},j=1,\ldots,n)$. We need to show that each $\alpha_j=0$, $\mod 2\pi$. In fact, through \eqref{eq:phase_formula}, for each $j$ we have
        \begin{equation*}
    \alpha_j = -\int\limits_0^1\Img\left(\widetilde w_j^T(\eta(t)) \pardt\widetilde{v}_j(\eta(t))\right)dt, \mod 2\pi,
    \end{equation*}
    and therefore
    \begin{equation*}
    \begin{split}
%    \alpha_j & = -\int\limits_0^1\Img\left(\widetilde w_j^T(\eta(t)) \pardt\widetilde{v}_j(\eta(t))\right)\, dt, \mod 2\pi \\
%    & = -\int\limits_0^\frac{1}{2}\Img\left(\widetilde w_j^T(\eta(t)) \dot{\widetilde{v}}_j(\eta(t))\dot\eta(t)\right)\, dt -\int\limits_\frac{1}{2}^1\Img\left(\widetilde w_j^T(\eta(t)) \dot{\widetilde{v}}_j(\eta(t))\dot\eta(t)\right)\, dt, \mod 2\pi \\
%    & = -\int\limits_0^\frac{1}{2}\Img\left(\widetilde w_j^T(\eta(t)) \dot{\widetilde{v}}_j(\eta(t))\dot\eta(t)\right)\, dt +\int\limits_\frac{1}{2}^0\Img\left(\widetilde w_j^T(\eta(1-t)) \dot{\widetilde{v}}_j(\eta(1-t))\dot\eta(1-t)\right)\, dt, \mod 2\pi \\
%    & = -\int\limits_0^\frac{1}{2}\Img\left(\widetilde w_j^T(\eta(t)) \dot{\widetilde{v}}_j(\eta(t))\dot\eta(t)\right)\, dt +\int\limits_0^\frac{1}{2}\Img\left(\widetilde w_j^T(\eta(t)) \dot{\widetilde{v}}_j(\eta(t))\dot\eta(t)\right)\, dt, \mod 2\pi \\
%    & = 0,\mod 2\pi.
%    \end{split}
    \alpha_j & = -\int\limits_0^\frac{1}{2}\Img\left(\widetilde w_j^T(\eta(t)) \dot{\widetilde{v}}_j(\eta(t))\dot\eta(t)\right) dt -\int\limits_\frac{1}{2}^1\Img\left(\widetilde w_j^T(\eta(t)) \dot{\widetilde{v}}_j(\eta(t))\dot\eta(t)\right) dt\\
    & = -\int\limits_0^\frac{1}{2}\Img\left(\widetilde w_j^T(\eta(t)) \dot{\widetilde{v}}_j(\eta(t))\dot\eta(t)\right) dt +\int\limits_\frac{1}{2}^0\Img\left(\widetilde w_j^T(\eta(1-t)) \dot{\widetilde{v}}_j(\eta(1-t))\dot\eta(1-t)\right) dt\\
    & = -\int\limits_0^\frac{1}{2}\Img\left(\widetilde w_j^T(\eta(t)) \dot{\widetilde{v}}_j(\eta(t))\dot\eta(t)\right) dt +\int\limits_0^\frac{1}{2}\Img\left(\widetilde w_j^T(\eta(t)) \dot{\widetilde{v}}_j(\eta(t))\dot\eta(t)\right) dt\\
    & = 0, \mod 2\pi.
    \end{split}
    \end{equation*}
    \,
\end{proof}
With the help of Lemma \ref{lem:back_and_forth} we can now prove that, taking a $2\times 2$ function on a loop containing two GCPs, there is no exchange of eigenvalues. In particular, this means that two GCPs can go undetected when the eigenvalues are continued around a loop that contains the two coalescing points.
\begin{thm}\label{thm:2GCPs_2x2}
    Let $A\in\cont^k(\Omega,\Ctxt)$, $k\ge1$, with $\Omega$ an open and convex subset of $\R^2$. Let $\xi_0, \xi_1\in\Omega$ be two distinct GCPs, and suppose that $A(\xi)$ has distinct eigenvalues for all $\xi\in\Omega\setminus\{\xi_0,\xi_1\}$. Let $\Gamma$ be a smooth Jordan curve enclosing both $\xi_0$ and $\xi_1$, parametrized by $\gamma$ in such way that both $\gamma(\cdot)$ and $A(\gamma(\cdot))$ are $1$-periodic. Finally, let $A(\gamma(t))=V(t)\Lam(t)V\inv(t)$ be a $\cont^k$ eigendecomposition of $A$ satisfying \eqref{eq:realP} and \eqref{eq:imgP}.
Then, for all $t\in\R$, we have:
\begin{enumerate}[(i)]
 \item $\lam_1(t+1) = \lam_1(t)\quad\text{ and }\quad\lam_2(t+1)= \lam_2(t)$;
 \item $V\inv(t)V(t+1) = \bmat{e^{i\alpha_1} & 0 \\ 0 & e^{i\alpha_2}}$, with $\alpha_1+\alpha_2=0,\mod 2\pi$.
\end{enumerate}
\end{thm}
\begin{proof}
    \begin{figure}
        \centering
        \includegraphics[width=0.5\textwidth]{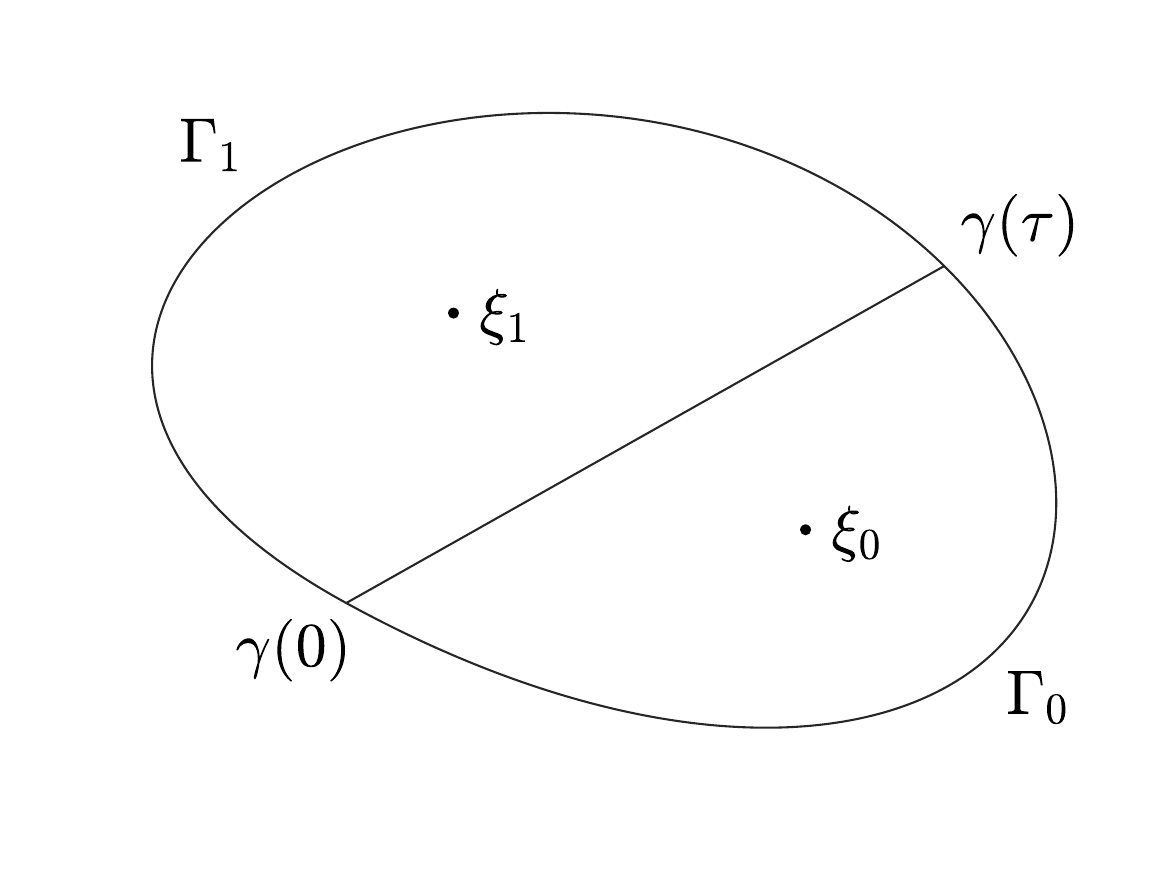}
        \caption{Two distinct GCPs inside the interior of the region bounded by the closed curve $\Gamma$ are separated by a line segment that does not contain any of the points.}
        \label{fig:patata}
    \end{figure}
    Consider a line segment joining two points on $\Gamma$ and separating $\xi_0$ and $\xi_1$ as in Figure \ref{fig:patata}. The segment does not contain $\xi_0$ and $\xi_1$, and we can assume that its end points are given by $\gamma(0)$ and $\gamma(\tau)$ for some $\tau\in(0,1)$. We can parametrize the segment by $\ell(z)=(\gamma(0)-\gamma(\tau))z+\gamma(\tau)$, $z\in[0,1]$, in such way that $\ell(0)=\gamma(\tau)$ and $\ell(1)=\gamma(0)$. Now consider the path
    \begin{equation}
        \zeta(s)=
        \begin{cases}
            \gamma(s), & s\in[0,\tau] \\
            \ell(s-\tau), & s\in[\tau,\tau+1] \\
            \ell(2+\tau-s), & s\in[\tau+1,\tau+2] \\
            \gamma(s-2), & s\in[\tau+2,3]
        \end{cases}
    \end{equation}
    Let $\zeta_0$ and $\zeta_1$ be the restrictions of $\zeta$ to $[0,\tau+1]$ and $[\tau+1,3]$, respectively. Assume, without loss of generality, that $\zeta_0$ parametrizes the simple closed curve $\Gamma_0$ that encircles $\xi_0$, and similarly for $\zeta_1$, $\xi_1$ and $\Gamma_1$, see Figure \ref{fig:patata}.  
    Observe that $\zeta_0$ and $\zeta_1$ are parametrized as piecewise smooth curves, and thus an eigendecomposition of $A$ along either of them will only be piecewise smooth: smooth along the portions of the original $\gamma$ and along the line segment, but with breakpoints and just continuity at the values $\gamma(\tau)$ and $\gamma(0)$, respectively. With this in mind, now consider $A(\zeta(3\sigma))$, $\sigma\in[0,1]$, and let $A(\zeta(3\sigma))=W(\sigma)D(\sigma)W\inv(\sigma)$ be a piecewise $\cont^k$ eigendecomposition of $A(\zeta(3\sigma))$ satisfying \eqref{eq:realP} and \eqref{eq:imgP} (except for the two breakpoints noted above) and also $W(0)=V(0)$. By construction, we have $V(\tau)=W(\tau/3)$. Because of Lemma \ref{lem:back_and_forth}, we have $W(\tau/3)=W((\tau+2)/3)$, and therefore also $V(\tau)=W((\tau+2)/3)$ and $V(1)=W(1)$. Now note that, because of Theorem \ref{thm:phasesum_2x2}, we have 
    \begin{equation*}
    \begin{array}{l}
    W\inv(0)W((\tau+1)/3)=\bmat{0 & e^{i\beta_{0,2}} \\ e^{i\beta_{0,1}} & 0}, \text{ with } \beta_{0,1}+\beta_{0,2}=\pi, \mod2\pi, \medskip\\
    W\inv((\tau+1)/3)W(1)=\bmat{0 & e^{i\beta_{1,2}} \\ e^{i\beta_{1,1}} & 0}, \text{ with } \beta_{1,1}+\beta_{1,2}=\pi, \mod2\pi.
    \end{array}
    \end{equation*}
    Combining all these identities, we finally have
    \begin{equation*}
    \begin{split}
        V\inv(0)V(1) & =W\inv(0)W\inv((\tau+1)/3)W((\tau+1)/3)W(1)= \\
        & =\bmat{0 & e^{i\beta_{0,2}} \\ e^{i\beta_{0,1}} & 0}\bmat{0 & e^{i\beta_{1,2}} \\ e^{i\beta_{1,1}} & 0}=\bmat{e^{i(\beta_{0,2}+\beta_{1,1})} & 0 \\ 0 & e^{i(\beta_{0,1}+\beta_{1,2})}}.
    \end{split}
    \end{equation*}
    The proof is completed by setting $\alpha_1\dfn\beta_{0,2}+\beta_{1,1}, \alpha_2\dfn\beta_{0,1}+\beta_{1,2}$ and noting that we have $\alpha_1+\alpha_2=0,\mod 2\pi$.
\end{proof}
Below, we give two examples to illustrate Theorem \ref{thm:2GCPs_2x2}.  In particular, in the first example we show that two GCPs may go undetected. In the second example, instead, monitoring the phases accrued by the eigendecomposition along a loop does betray the presence of two GCPs, even though periodicity of the eigenvalues gives no clue.

\begin{exm}\label{exm:phase_zero} Let $$A(x,y)=\bmat{0 & 1\\ x^2+y^2-\eps +iy & 0},$$ where $x,y\in\R$ and $\eps>0$.
For any $\eps>0$, there are exactly two distinct points of coalescence for the eigenvalues of $A$ in $\R^2$, located at
\begin{equation}\label{eq:exm1_GCPs}
\xi_{\pm}=\big(\pm\sqrt{\eps},0\big).
\end{equation}
Both points are GCPs since (see Definition \ref{def:cuspPoint})
\begin{equation*}
    F(x,y)=4\bmat{x^2+y^2-\eps \\ y},\quad DF(x,y)=
    4\bmat{2x & 2y \\ 0 & 1},
\end{equation*}
and hence $\xi_{\pm}$ are regular zeros for $F$.
Now consider the circle $C$ centered at the origin, parametrized by
\begin{equation}\label{eq:circles_param}
    \gamma(t)=\rho\bmat{
    \cos(2\pi t) \\ \sin(2\pi t)},\quad t\in[0,1],
\end{equation}
where $\rho$ is chosen large enough so that $C$ encloses both GCPs.
Let $A(\gamma(t))=V(t)\Lam(t)V\inv(t)$ be a $\cont^k$ eigendecomposition of $A$ along $C$ satisfying \eqref{eq:realP} and \eqref{eq:imgP} for $t\in[0,1]$. Then:
\begin{enumerate}[(i)]
    \item $\lam_1(1)=\lam_1(0),\ \lam_2(1)=\lam_2(0)$,
    \item $V(1)=V(0)$, i.e. one can take $\alpha_1=\alpha_2=0$ in item \emph{(ii)} of Theorem \ref{thm:2GCPs_2x2}.
\end{enumerate}
To understand why, observe the following:
\begin{enumerate}[(1)]
    \item The periodicity of the eigenvalues $\lam_1(\cdot)$, $\lam_2(\cdot)$ follows
    from a direct computation (and see also item \emph{(i)} of Theorem \ref{thm:2GCPs_2x2});
    \item By separating $\xi_+$ and $\xi_-$ with a vertical segment, as shown in Figure~\ref{fig:examples}, the phase accumulated by each eigenvector of $V$ around $C$ equals the sum of the phases accumulated along the paths $\Gamma_+$ and $\Gamma_-$ (see the proof of Theorem~\ref{thm:2GCPs_2x2});
    \item According to Remark~\ref{rem:special2x2}, the phases accrued along $\Gamma_+$ can be taken to be either $\pi/2$ or $-\pi/2$;
    \item Since $\overline{A(x,y)} = A(-x,-y)$ and $\Gamma_-$ is the image of $\Gamma_+$ under a $\pi$-rotation about the origin, then, in lieu of Remark~\ref{rem:complexConj_negativePhase}, the phases accrued along $\Gamma_-$ are the negatives of those accrued along $\Gamma_+$;
    \item Consequently, the total phase accrued by each eigenvector around $C$ is zero.
\end{enumerate}
Therefore, neither periodicity of the eigenvalues nor monitoring of the phases betrays the presence of the two GCPs.
\end{exm}

\begin{figure}\label{fig:examples}
    \centering
    \includegraphics[width=0.5\linewidth]{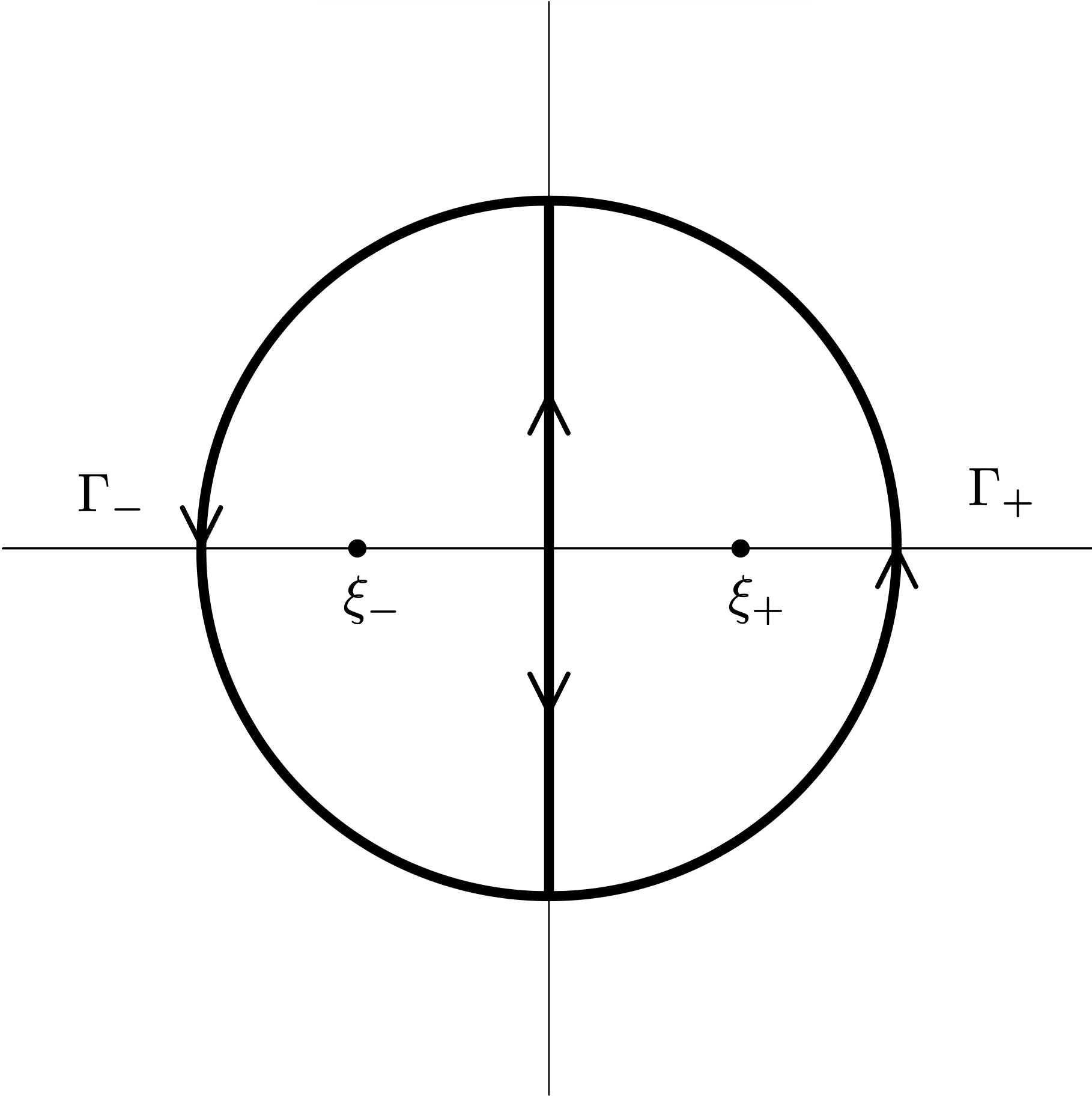}
    \caption{Reference figure for Examples \ref{exm:phase_zero} and \ref{exm:phase_pi}; $\Gamma_+$ is the right half-circle and $\Gamma_-$ is the left one, both traversed counterclockwise.}
\end{figure}
\begin{exm}\label{exm:phase_pi} Let $$A(x,y)=\bmat{0 & 1\\ xy-\eps +i (x^2-y^2-\eps)& 0},$$ where $x,y\in\R$ and $\eps>0$. Here too, for any $\eps>0$, there are exactly two distinct points of coalescence for the eigenvalues $A$ in $\R^2$, given by $\left(\pm \sqrt{\eps\frac{1 + \sqrt{5}}{2}},
\pm\sqrt{\eps\frac{2}{1 + \sqrt{5}} }\right)$. Again, we consider the circle $C$ parametrized by \eqref{eq:circles_param}, with $\rho$ large enough so that $C$ encloses both GCPs. Then, the $\cont^k$ eigendecomposition $A(\gamma(t))=V(t)\Lam(t)V\inv(t)$ satisfying \eqref{eq:realP} and \eqref{eq:imgP} satisfies:
\begin{enumerate}[(i)]
    \item $\lam_1(1)=\lam_1(0),\ \lam_2(1)=\lam_2(0)$,
    \item $V(1)=-V(0)$, i.e. one can take $\alpha_1=\alpha_2=\pi$ in item \emph{(ii)} of Theorem \ref{thm:2GCPs_2x2}.
\end{enumerate}
To see why, we reason as in Example~\ref{exm:phase_zero}, with the exception that, in item (4), $A(x,y) = A(-x,-y)$. As a result, by Remark~\ref{rem:complexConj_negativePhase}, the phases accrued along $\Gamma_-$ are the same as those accrued along $\Gamma_+$, and the total phase accumulated by each eigenvector around $C$ is $\pi$.
\end{exm}

So far, the effect of GCPs on eigenvectors under their continuation along Jordan curves has been analyzed only for $2 \times 2$ matrix-valued functions.
%}
The following theorems address the case of $n\times n$ matrix functions.  Taken together, the next -and the previous- results tell us that by monitoring the periodicity of the eigenvalues along a loop, possibly supplemented by monitoring of the phases, may allow us to predict having GCPs inside the loop. Below, we first consider the case of having a single isolated GCP. Without loss of generality, we will assume that the coalescing eigenvalues are labeled as $\lam_1$ and $\lam_2$.

\begin{thm}\label{thm:singleGCP_nxn} Let $A\in\cont^k(\Omega,\Cnxn)$, $k\ge1$, where $n\ge 3$, and  $\Omega$ be a convex subset of $\R^2$. Let $\xi_0\in\Omega$ be a GCP, and suppose that $A(\xi)$ has distinct eigenvalues everywhere else in $\Omega$.
Let $\Gamma$ be a smooth Jordan curve encircling the point $\xi_0$, parametrized by $\gamma$ in such way that both $\gamma(\cdot)$ and $A(\gamma(\cdot))$ are $1$-periodic. Finally, let $A(\gamma(t))=V(t)\Lam(t)V\inv(t)$ be a $\cont^k$ eigendecomposition of $A$ satisfying \eqref{eq:realP} and \eqref{eq:imgP}.
Then, the eigenvalues of $A(\gamma(\cdot))$ can be labeled so that, for all $t\in\R$, we have:
\begin{enumerate}[(i)]
    \item $\lam_1(t+1) = \lam_2(t),\ \lam_2(t+1)=\lam_1(t)$;
    \item $\lam_j(t+1) = \lam_j(t)$ for all $j\ge 3$;
    \item $V^{-1}(t) V(t+1) = 
\bmat{
0 & e^{i\alpha_2}  & 0 & \cdots & 0 \\
e^{i\alpha_1} & 0 & 0 & \cdots & 0 \\
0 & 0 & e^{i\alpha_3} & \cdots & 0 \\
\vdots & \vdots & \vdots & \ddots & \vdots \\
0 & 0 & 0 & \cdots & e^{i\alpha_n}
}$,\ with $\sum\limits_{j=1}^n\alpha_j=\pi, \mod 2\pi$.
\end{enumerate}
Moreover, for a family of loops $\Gamma_s$ shrinking down to $\xi_0$ as in Theorem \ref{thm:loops2point_GCP},
we have, as $s\rightarrow 0$, 
\begin{equation*}
\begin{array}{ll}
\alpha_j(s)=\dfrac{\pi}{2}+\bigo(\sqrt{s}), \mod\pi, & \text{ for } j=1,2\,, \medskip\\
\alpha_j(s)=\bigo(s), \mod 2\pi, & \text{ for } j\ge 3\,.
\end{array}
\end{equation*}
\end{thm}
\begin{proof}
The proof follows easily from our previous results, and we therefore just outline it, omitting some technical details. Let $R\subset\Omega$ be a rectangular region containing $\xi_0$ in its interior. By Theorem~\ref{thm:blockdec}, the matrix $A$ admits a smooth block decomposition
\begin{equation*}
T^{-1}(\xi)A(\xi)T(\xi) = \begin{bmatrix} E_1(\xi) & 0 \\ 0 & E_2(\xi) \end{bmatrix},
\end{equation*}
for all $\xi \in R$, where $E_1 \in \cont^k(R, \C^{2\times 2})$ has eigenvalues $\lambda_1(\xi), \lambda_2(\xi)$ that coalesce only at the GCP $\xi_0$, while $E_2 \in \cont^k(R, \C^{(n-2)\times(n-2)})$ has eigenvalues that remain distinct throughout $R$.

Next, consider a homotopy that continuously deforms $\Gamma$ into a circle $C$ centered at $\xi_0$, fully contained in $R$, and parametrized by $\tilde{\gamma}$. Applying Theorems~\ref{thm:mainres_2x2} and \ref{thm:loops2point_GCP} to $E_1$, and Theorems~\ref{thm:simple_eig} and \ref{thm:loops2point_simple_eig} to $E_2$, we conclude that properties \emph{(i)} and \emph{(ii)} hold for the eigenvalues of $E_1(\tilde\gamma(t))$ and $E_2(\tilde\gamma(t))$, respectively. These properties carry over from $C$ to $\Gamma$ by continuity of the deformation, and hence extend to the eigenvalues $\lambda_1(t), \lambda_2(t)$ and $\lambda_3(t), \ldots, \lambda_n(t)$ of $A(\gamma(t))$.

Regarding \emph{(iii)}, the structure of $V^{-1}(t)V(t+1)$ follows from \emph{(i)} and \emph{(ii)}, and the fact that the $\alpha_j$’s are constant is established as in part \emph{(i)} of Lemma~\ref{lem:alpha_trace_det}. Finally, the identity $\sum_{j=1}^n \alpha_j = \pi \mod 2\pi$ also carries over by continuity of the deformation, and a similar argument yields the asymptotics for the $\alpha_j$'s as $\Gamma$ shrinks down to $\xi_0$.
\end{proof}
\color{black}

We now examine $n \times n$ matrix functions, with $n\ge 3$, that exhibit two distinct GCPs, each associated with a different pair of eigenvalues. Proofs are omitted, as they proceed along the same lines as the proofs of Theorems \ref{thm:2GCPs_2x2} and \ref{thm:singleGCP_nxn}.
\begin{thm}\label{thm:twoGCP_nxn_disj} Let $A\in\cont^k(\Omega,\Cnxn)$, $k\ge1$, where $n\ge 5$, and  $\Omega$ be a convex subset of $\R^2$. Assume that the eigenvalues of $A$ can be labeled so that, for all $\xi\in\Omega$, they are grouped into three mutually disjoint sets
\begin{equation*}
\{\lam_1(\xi),\lam_2(\xi)\},\quad \{\lam_3(\xi),\lam_4(\xi)\},\quad \{\lam_5(\xi),\ldots,\lam_n(\xi)\}.
\end{equation*}
Suppose that $\xi_0,\xi_1\in\Omega$ are two distinct GCPs where
\begin{equation*}
    \lam_1(\xi_0)=\lam_2(\xi_0),\quad \lam_3(\xi_1)=\lam_4(\xi_1),
\end{equation*}
and that $A(\xi)$ has distinct eigenvalues everywhere else in $\Omega$.
Let $\Gamma$ be a smooth Jordan curve encircling both points $\xi_0$ and $\xi_1$, parametrized by $\gamma$ in such way that both $\gamma(\cdot)$ and $A(\gamma(\cdot))$ are $1$-periodic. Finally, let $A(\gamma(t))=V(t)\Lam(t)V\inv(t)$ be a $\cont^k$ eigendecomposition of $A$ satisfying \eqref{eq:realP} and \eqref{eq:imgP}.
Then, the eigenvalues of $A(\gamma(\cdot))$ can be labeled so that, for all $t\in\R$, we have:
\begin{enumerate}[(i)]
    \item $\lam_1(t+1) = \lam_2(t),\ \lam_2(t+1)=\lam_1(t)$;
    \item $\lam_3(t+1) = \lam_4(t),\ \lam_4(t+1)=\lam_3(t)$;
    \item $\lam_j(t+1) = \lam_j(t)$ for all $j\ge 5$;
    \item $V^{-1}(t) V(t+1) = \bmat{
0 & e^{i\alpha_2} & 0 & 0 & 0 & \cdots & 0 \\
e^{i\alpha_1} & 0 & 0 & 0 & 0 & \cdots & 0 \\
0 & 0 & 0 & e^{i\alpha_4} & 0 & \cdots & 0 \\
0 & 0 & e^{i\alpha_3} & 0 & 0 & \cdots & 0 \\
0 & 0 & 0 & 0 & e^{i\alpha_5} & \cdots & 0 \\
\vdots & \vdots & \vdots & \vdots & \vdots & \ddots & \vdots \\
0 & 0 & 0 & 0 & 0 & \cdots & e^{i\alpha_n}
}$,\ with $\sum\limits_{j=1}^n\alpha_j=0, \mod 2\pi$.
\end{enumerate}
\qed
\end{thm}
\begin{thm}\label{thm:twoGCP_nxn_consec} Let $A\in\cont^k(\Omega,\Cnxn)$, $k\ge1$, where $n\ge 4$ and  $\Omega$ be a convex subset of $\R^2$. Assume that the eigenvalues of $A$ can be labeled so that, for all $\xi\in\Omega$, they are grouped into two mutually disjoint sets
\begin{equation*}
\{\lam_1(\xi),\lam_2(\xi)\,\lam_3(\xi)\},\quad \{\lam_4(\xi),\ldots,\lam_n(\xi)\}.
\end{equation*}
Suppose that $\xi_0,\xi_1\in\Omega$ are two distinct GCPs where
\begin{equation*}
    \lam_1(\xi_0)=\lam_2(\xi_0),\quad \lam_2(\xi_1)=\lam_3(\xi_1),
\end{equation*}
and that $A(\xi)$ has distinct eigenvalues everywhere else in $\Omega$.
Let $\Gamma$ be a smooth Jordan curve encircling both points $\xi_0$ and $\xi_1$, parametrized by $\gamma$ in such way that both $\gamma(\cdot)$ and $A(\gamma(\cdot))$ are $1$-periodic. Finally, let $A(\gamma(t))=V(t)\Lam(t)V\inv(t)$ be a $\cont^k$ eigendecomposition of $A$ satisfying \eqref{eq:realP} and \eqref{eq:imgP}.
Then, the eigenvalues of $A(\gamma(\cdot))$ can be labeled so that, for all $t\in\R$, we have:
\begin{enumerate}[(i)]
    \item $\lam_1(t+1) = \lam_2(t),\ \lam_2(t+1)=\lam_3(t),\ \lam_3(t+1)=\lam_1(t)$;
    \item $\lam_j(t+1) = \lam_j(t)$ for all $j\ge 4$;
    \item $V^{-1}(t) V(t+1) = \bmat{
0 & e^{i\alpha_2} & 0 & 0 & \cdots & 0 \\
0 & 0 & e^{i\alpha_3} & 0 & \cdots & 0 \\
e^{i\alpha_1} & 0 & 0 & 0 & \cdots & 0 \\
0 & 0 & 0 & e^{i\alpha_4} & \cdots & 0 \\
\vdots & \vdots & \vdots & \vdots & \ddots & \vdots \\
0 & 0 & 0 & 0 & \cdots & e^{i\alpha_n}
}$,\ with $\sum\limits_{j=1}^n\alpha_j=0, \mod 2\pi$.
\end{enumerate}
\qed
\end{thm}
\begin{rem}
We refrain from stating a full generalization of Theorems \ref{thm:singleGCP_nxn}, \ref{thm:twoGCP_nxn_disj}, and \ref{thm:twoGCP_nxn_consec} to account for an arbitrary number of GCPs, as this would be unnecessarily cumbersome. What is important to emphasize is that, when computing an eigendecomposition $A(t) = V(t)D(t)V^{-1}(t)$ satisfying \eqref{eq:realP} and \eqref{eq:imgP} along a loop containing several GCPs, each GCP contributes to the pattern and phases of the entries of $V^{-1}(t)V(t+1)$ in a (partially) predictable way.
\end{rem}
\section{Conclusions}
In this work, we have considered smooth complex valued matrix functions depending on 2 (real) parameters: $A\in \cont^k(\Omega, \Cnxn)$, $\Omega$ a convex subset of  $\R^2$.  Unlike the better understood case of Hermitian $A$, in this work we focused on general, unstructured, $A$.  Our goal has been to rigorously characterize how to identify parameter values where eigenvalues of $A$ coalesce, which we call cuspidal points.  Our main result has been to show that, if we take a loop in parameter space enclosing a generic cuspidal point (GCP), and with the eigenvalues being otherwise distinct inside and along the loop, then by looking at the periodicity of the eigenvalues and to the phase accumulation of the eigenvectors along the loop, one is able to detect generic cuspidal points.  For this result to hold, the function $V$ of eigenvectors of $A$ along the loop must be taken smooth and satisfy very specific conditions for the real and imaginary parts of the diagonal of $V^{-1}\dot V$; see our relations
\eqref{eq:realP} and \eqref{eq:imgP}.
We further extended our results to the case of two (or more) GCPs enclosed by a loop and showed that in this case the eigenvalues' periodicity does not betray the GCPs, but the phases accumulations may.  Finally, we also gave results on the asymptotic values of the phase accumulation when there is, or not, a GCP inside a loop shrinking to a point (the GCP itself, if there is one inside the loop).

The study of robust algorithmic techniques to numerically locate GCPs will be part of future work, as it will be the statistical study of the distribution and density (in parameter space) of GCPs for two-parameter families of random unstructured matrices.

\appendix
\section{Periodicity of the composition of functions}\label{appendix:period}
Let $f:\R^2\to\R$ be a smooth function. Let $\mathbb{S}^1=\{(x,y)\in\R^2: x^2+y^2=1\}$ be the unit circle. Let $\gamma$ be the following parametrization of $\mathbb{S}^1$:
\begin{equation}\label{def:standardParam}
 \gamma(t)=\bmat{\cos(2\pi t) \\ \sin(2\pi t)},\ t\in\R.
\end{equation}
We say that a function $\varphi$ of one real variable $t$ is $\tau$-periodic if
\begin{equation}\label{def:periodic}
\varphi(t+\tau)=\varphi(t)
\end{equation}
for all $t\in\R$, with $\tau>0$. If $\tau$ is the smallest strictly positive number for which \eqref{def:periodic} holds, we say that $\tau$ is the \emph{minimal positive period} of $\varphi$. The function $\gamma$ defined in \eqref{def:standardParam} has minimal positive period 1.

We know that the minimal positive period of the composition $f(\gamma(t))$ could be strictly smaller than 1. For instance, consider $f(x,y)=xy$.

\begin{thm}\label{thm:no_subperiod}
Given any non-constant $\cont^k$ function $f:\R^2\to\R$, $k\ge 0$, there exists a $\cont^\infty$  function $\varphi:\R\to\R$ such that:
\begin{enumerate}[(i)]
 \item $\varphi(0)=0$, $\varphi(1)=1$;
 \item $f\circ\gamma\circ\varphi$ is $1$-periodic;
 \item 1 is the minimal positive period of $f\circ\gamma\circ\varphi$.
\end{enumerate}
\end{thm}

We will deduce Theorem \ref{thm:no_subperiod} from the following

\begin{lem}\label{lem:change_of_var}
 Given any non-constant $1$-periodic $\cont^k$ function $g:\R\to\R$, $k\ge 0$, there exists a $\cont^\infty$ function $\varphi:\R\to\R$ such that:
\begin{enumerate}[(i)]
 \item $g\circ\varphi$ is $1$-periodic;
 \item 1 is the least positive period of $g\circ\varphi$.
\end{enumerate}
\end{lem}
\begin{proof}
Since $g$ is non-constant, there exist $t^*\in(0,1)$ and $\eps>0$ such that
\begin{equation}
 g(t)\neq g(0) \text{ for all } t \text{ in } I=[t^*-\eps, t^*+\eps].
\end{equation}
There exists a function $\varphi:\R\to\R$ such that:
\begin{enumerate}[i)]
 \item $\varphi(t)\in I$ for all $t\in[1/3,2/3]$;
 \item $\varphi(t+1)=\varphi(t)+1$ for all $t$ in $\R$;
 \item $\varphi$ is $\cont^\infty$ on $\R$.
\end{enumerate}

One recipe to construct such a function $\varphi$ is as follows: 

\begin{enumerate}[(1)]
 \item Consider the piecewise linear function $\varphi_0$ depicted in Figure \ref{fig:phi_p};
 \item extend $\varphi_0$ to $\R$ in such way that $\varphi(t+1)=\varphi(t)+1$ for all $t$ in $\R$;
 \item take a $\cont^\infty$ symmetric positive mollifier $\rho_\delta$ supported on $[-\delta,\delta]$, with $\delta>0$ sufficiently small;
 \item define $\varphi=\varphi_0\ast\rho_\delta$, where $\ast$ is the convolution operation.
\end{enumerate}

A simple check shows that $g\circ\varphi$ is $1$-periodic. Any other period $\tau\in(0,1)$ of $g\circ\varphi$ must be rational. If not, $g$ would be constant. Without loss of generality, we can restrict our attention to rational periods of the form $1/q$, with $q\ge2$ integer. Let $q\ge2$ be an integer. There exists $k\in\{1, 2, \ldots, q-1\}$ integer such that $k/q\in[1/3,2/3]$, which implies $g(\varphi(k/q))\ne g(0)$. This means that $g\circ\varphi$ cannot have period $1/q$.

\begin{figure}[ht]
\begin{center}
\includegraphics[scale=.5]{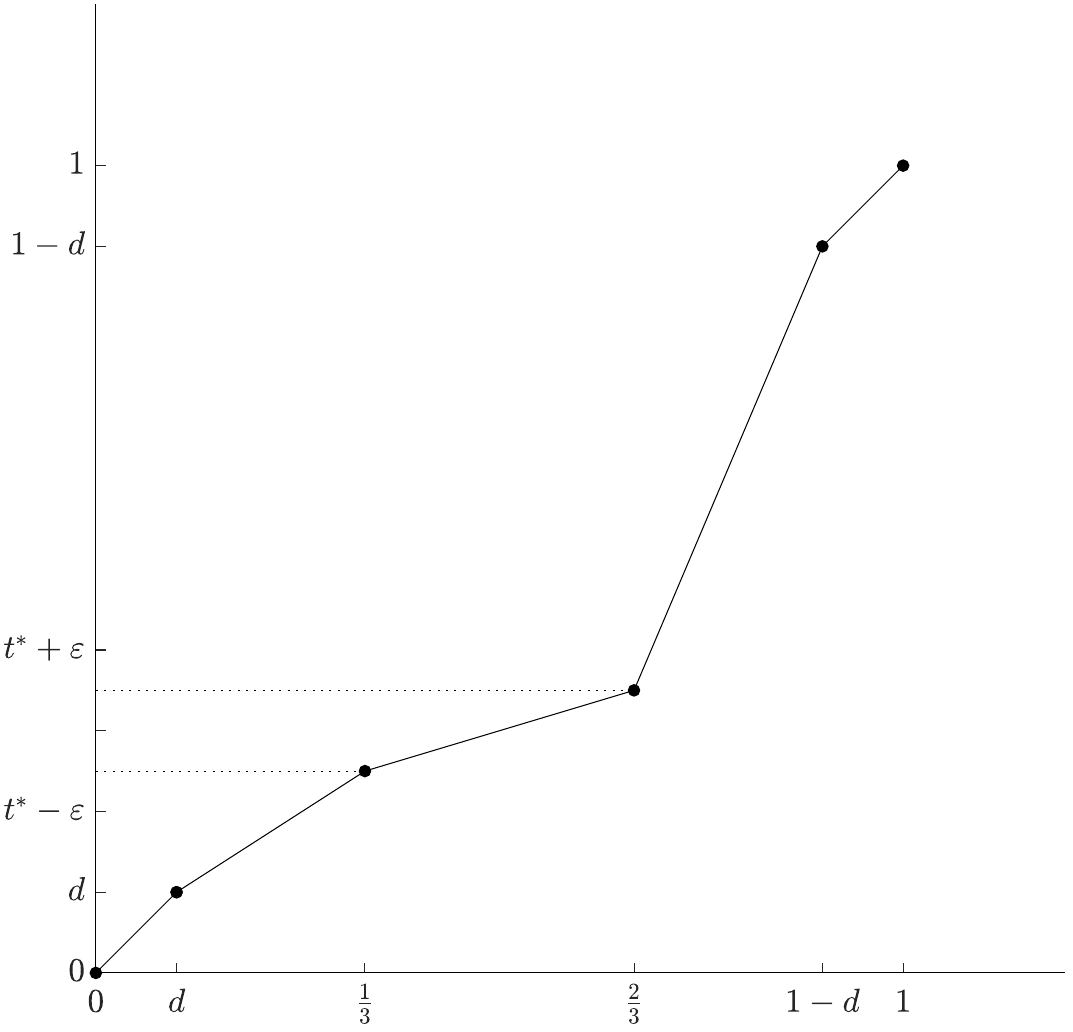}
\caption{Graph of $\varphi_0$.}
\label{fig:phi_p}
\end{center}
\end{figure}
\end{proof}

\begin{proof}[Proof of Theorem \ref{thm:no_subperiod}]
 Apply Lemma \ref{lem:change_of_var} to $g=f\circ\gamma$.
\end{proof}
\begin{rem}\label{rem:perCompFun_complex}
Note that the conclusion of Theorem \ref{thm:no_subperiod} also holds for functions $f:\R^2 \to \C$, by applying the previous argument to either the real or imaginary part of $f$.
\end{rem}

\bibliographystyle{siamplain}

\begin{thebibliography}{99}

\bibitem{BerkoPar}
G. Berkolaiko and A. Parulekar.
\newblock
Locating Conical Degeneracies in the Spectra of Parametric Self-adjoint Matrices.
\newblock{\em{{SIAM J. Matrix Anal. Appl.}}}, 42:224-242 (2020).

\bibitem{BerkoZel}
G. Berkolaiko and I. Zelenko.
\newblock
Morse inequalities for ordered eigenvalues of generic self-adjoint families.
\newblock{\em{Inventiones mathematicae}}, 238:283 - 330 (2023).
%\color{black}

\bibitem{BerryUzdin}
M.V. Berry, and R. Uzdin.
\newblock
Slow non-Hermitian cycling: exact solutions and the Stokes phenomenon.
\newblock{\em{J. Phys. A: Math. Theor.}}, 44 (2011) 435303 (26pp) doi:10.1088/1751-8113/44/43/435303

\bibitem{Dembowski}
C. Dembowski, H.D. Gr\"af, H. L. Harney, A., Heine, W. D., Heiss, H.  and A. Richter.
\newblock 
Experimental Observation of the Topological Structure of Exceptional Points.
\newblock{\em{Physical Review Letters}}, 86-5: 787--790, 2001.

\bibitem{Deng}
Deng, Kuangyin and Li, Xin and Flebus, Benedetta.
\newblock Exceptional points as signatures of dynamical magnetic phase transitions.
\newblock{\em{Physical Review B}} 107, L100402 (2023).

\bibitem{DE99}
Dieci, Luca and Eirola, Timo.
\newblock On smooth decompositions of matrices.
\newblock {\em SIAM Journal on Matrix Analysis and Applications}, 20(3):800--819, 1999.

\bibitem{DiPaPu_simax2013}
L.~Dieci, A.~Papini, and A.~Pugliese,
Approximating coalescing points for eigenvalues of Hermitian matrices of three parameters,
\emph{SIAM J. Matrix Anal. Appl.}, 34(2) (2013), 519--541.

\bibitem{DiPuSVD}
Dieci, Luca and Pugliese, Alessandro.
\newblock Two-parameter SVD: Coalescing singular values and periodicity. \newblock{\em{SIAM J. Matrix Analysis}} 31(2): 375-403, 2009. 

\bibitem{DP12herm}
Dieci, Luca and Pugliese, Alessandro.
\newblock Hermitian matrices depending on three parameters: coalescing
  eigenvalues.
\newblock {\em Linear Algebra Appl.}, 436(11):4120--4142, 2012.

\bibitem{Doppler}
Doppler, J\"org and Mailybaev, Alexei and B\"ohm, Julian and Kuhl, Ulrich and Girschik, Adrian and Libisch, Florian and Milburn, Thomas and Rabl, Peter and Moiseyev, Nimrod and Rotter, Stefan.
\newblock Dynamically encircling an exceptional point for asymmetric mode switching.
\newblock {\em Nature}, 537,  2016.
{\tt doi = {10.1038/nature18605}}

\bibitem{Garrison}
J.C. Garrison, E.M. Wright.
\newblock Complex geometrical phases for dissipative systems.
\newblock{\em{Physics Letters A}}, Volume 128, N. 3-4, 177--181, 1988.

\bibitem{Hale}
J.~K.~Hale.
\newblock {\em Ordinary Differential Equations}.
\newblock Krieger Publishing Co., Malabar, 1980.

\bibitem{HS66}
P.~F.~Hsieh, Y.~Sibuya.
\newblock A global analysis of matrices of functions of several variables.
\newblock {\em Journal of Mathematical Analysis and Applications},
  14(2):332--340, 1966.
  
 \bibitem{Kato}
T.~Kato.
\newblock {\em {Perturbation theory for linear operators; 2nd ed.}}
\newblock Classics in Mathematics. Springer-Verlag, 1976.

\bibitem{Keller08}
J.~B.~Keller.
\newblock Multiple eigenvalues.
\newblock {\em Linear Algebra and its Applications}, 429(8):2209--2220, 2008.

\bibitem{Kirillov}
Oleg N. Kirillov.
\newblock
Exceptional and diabolical points in stability questions.
\newblock{\em{Fortschr. Phys.}} 61, No. 2 – 3, 205--224 (2013), DOI 10.1002/prop.201200068.

\bibitem{Ryu1}
J.-W. Ryu, S.-Y. Lee, and S. W. Kim.
\newblock
Analysis of multiple exceptional points related to three interacting eigenmodes
in a non-Hermitian Hamiltonian.
\newblock{\em Phys. Rev. A}, 85,
042101 (2012).

\bibitem{Ryu2}
S.Y. Lee, J.W. Ryu, S.W. Kim, and Y.Chung.
\newblock Geometric phase around multiple exceptional points.
\newblock {\em Phys. Rev. A} 85, 064103 (2012).

\bibitem{Ryu3}
J.W. Ryu, J.H. Han, C.H. Yi, M.J. Park, and H.C. Park.
\newblock 
Exceptional classifications of non-Hermitian systems.
\newblock {\em Communications Physics} 7, 109 (2024).

\bibitem{Bhandari}
J. Samuel and R. Bhandari.
\newblock General setting for Berry's phase.
\newblock{\em{Physical Review Letters}}, 60-23: 2339--2342, 1988.

\bibitem{Simon}
B.~Simon.
\newblock Holonomy, the Quantum Adiabatic Theorem, and Berry's Phase.
\newblock{\em {Phys. Rev. Lett.}}, 51-24, pp. 2167--2170 (1983).

\bibitem{VonNeuWig}
J.~von Neumann, E.~Wigner.
\newblock Über das Verhalten von Eigenwerten bei adiabatischen Prozessen.
\newblock {\em Physikalische Zeitschrift}, 30:467--470, 1929.

\bibitem{Younes2010}
Laurent Younes.
\newblock \emph{Shapes and Diffeomorphisms}.
\newblock Springer, 2010.

\end{thebibliography}

\end{document}